\documentclass[10pt]{article}
\usepackage{defs}
\usepackage[charsperline = 94, theme = grayscale]{jlcode}
\usepackage{subcaption}

\newcommand{\papertitle}{%
  Complexity of trust-region methods with potentially unbounded Hessian approximations for smooth and nonsmooth optimization
}

\hypersetup{
  pdftitle=\papertitle,
  pdfauthor={Geoffroy Leconte and Dominique Orban},
  pdfsubject={Nonsmooth Regularized Optimization},
  pdfkeywords={},
}

\usepackage{newunicodechar}
\newunicodechar{Δ}{\ensuremath{\Delta}}
\newunicodechar{ξ}{\ensuremath{\xi}}
\newunicodechar{ν}{\ensuremath{\nu}}
\newunicodechar{√}{\ensuremath{\surd}}
\newunicodechar{ρ}{\ensuremath{\rho}}
\newunicodechar{ₖ}{\ensuremath{_k}}

\title{\papertitle}
\author{%
  Geoffroy Leconte\footnote{%
    GERAD and Department of Mathematics and Industrial Engineering, Polytechnique Montr\'eal. E-mail: \href{mailto:geoffroy.leconte@polymtl.ca}{geoffroy.leconte@polymtl.ca}.
  }
  \and
  Dominique Orban\footnote{%
    GERAD and Department of Mathematics and Industrial Engineering, Polytechnique Montr\'eal. E-mail: \href{mailto:dominique.orban@gerad.ca}{dominique.orban@gerad.ca}.
  }
  \thanks{Research supported by an NSERC Discovery grant.}
}

\begin{document}
\maketitle
\thispagestyle{mytitlepage}

\begin{abstract}
  We develop a worst-case evaluation complexity bound for trust-region methods in the presence of unbounded Hessian approximations.
  We use the algorithm of \citet{aravkin-baraldi-orban-2022} as a model, which is designed for nonsmooth regularized problems, but applies to unconstrained smooth problems as a special case.
  Our analysis assumes that the growth of the Hessian approximation is controlled by the number of successful iterations.
  We show that the best known complexity bound of \(\epsilon^{-2}\) deteriorates to \(\epsilon^{-2/(1-p)}\), where \(0 \leq p < 1\) is a parameter that controls the growth of the Hessian approximation.
  The faster the Hessian approximation grows, the more the bound deteriorates.
  We construct an objective that satisfies all of our assumptions and for which our complexity bound is attained, which establishes that our bound is sharp.
  To the best of our knowledge, our complexity result is the first to consider potentially unbounded Hessians and is a first step towards addressing a conjecture of \citet{powell-2010} that trust-region methods may require an exponential number of iterations in such a case.
  Numerical experiments conducted in double precision arithmetic are consistent with the analysis.
\end{abstract}


\pagestyle{myheadings}

\section{Introduction}

We consider the nonsmooth regularized problem
\begin{equation}%
  \label{eq:nlp}
  \minimize{x \in \R^n} \ f(x) + h(x)
  \quad \st \ \ell \leq x \leq u,
\end{equation}
where \(\ell \in (\R \cup \{-\infty\})^n\), \(u \in (\R \cup \{+\infty\})^n\) with \(\ell \leq u\) componentwise, \(f: \R^n \to \R\) is continuously differentiable on an open set containing the feasible set \([\ell, u]\) of~\eqref{eq:nlp}, and \(h: \R^n \to \R \cup \{+\infty\}\) is proper and lower semicontinuous (lsc).
A component \(\ell_i = -\infty\) or \(u_i = +\infty\) indicates that \(x_i\) is unbounded below or above, respectively.
Both \(f\) and \(h\) may be nonconvex.
The nonsmooth regularizer \(h\) is often used to identify a local minimizer of \(f\) with desirable features, such as sparsity.

Algorithms used to solve~\eqref{eq:nlp} are often based on the proximal-gradient method \citep{lions-mercier-1979,fukushima-mine-1981}.
The algorithm that we consider here is the trust-region method (TR) of \citet{aravkin-baraldi-orban-2022}, which improves upon the proximal-gradient method by constructing a model of \(f\) and a model of \(h\) at each iteration in order to compute a step, in the spirit of traditional trust-region methods \citep{conn-gould-toint-2000}.
To the best of our knowledge, it is the only trust-region method for~\eqref{eq:nlp} that allows both \(f\) and \(h\) to be nonconvex, and that only assumes that \(h\) is proper lsc.
Typically, the model of \(f\) is a quadratic about the current iterate, and we denote its Hessian by \(B_k\); the latter may be the Hessian of \(f\) if it exists, or an approximation thereof.
TR was developed under the assumption that \(\{B_k\}\) remains bounded, a common, but sometimes restrictive, assumption.
A worst-case evaluation complexity bound for a stationarity measure to drop below \(\epsilon \in (0, 1)\) of \(O(\epsilon^{-2})\) results, which matches the best possible complexity bound in the smooth case, i.e., when \(h = 0\) \citep{cartis-gould-toint-2022}.

In the present paper, we examine the situation where \(\{B_k\}\) is allowed to grow unbounded.
We impose a bound on the growth of \(\|B_k\|\) in terms of the number of successful iterations that is slightly more restrictive than bounds used in smooth optimization to establish global convergence---see below.
Our tighter growth control, however, allows us to formalize a worst-case evaluation complexity bound, which we then show to be tight.
Specifically, we show that the best known complexity bound of \(O(\epsilon^{-2})\) deteriorates to \(O(\epsilon^{-2/(1-p)})\), where \(0 \leq p < 1\) is a parameter that controls the growth of \(\|B_k\|\).
To the best of our knowledge, this is the first formal worst-case analysis in the case of potentially unbounded \(B_k\).

A Julia implementation of TR is available as part of the RegularizedOptimization.jl package \citep{baraldi-orban-regularized-optimization-2022}.
Our findings also apply to Algorithm TRDH of \citet{leconte-orban-2024}, which is similar to TR, but uses diagonal Hessian approximations to compute a step without recourse to a subproblem solver.

Unbounded, or potentially unbounded, Hessians are not uncommon in applications.
A prime example is interior-point methods for bound-constrained optimization.
Consider the minimization of a twice differentiable objective \(\phi: \R^n \to \R\) subject to simple bounds \(x \geq 0\).
Primal interior-point methods \citep{fiacco-mccormick-1968} consist in applying Newton's method to a sequence of log-barrier subproblems whose objective is \(\phi(x) - \mu \sum_i \log(x_i)\) where \(\mu > 0\) is a barrier parameter that is eventually driven to zero.
Such methods maintain \(x > 0\) implicitly but the barrier objective Hessian is \(\nabla^2 \phi(x) + \mu X^{-2}\), where \(X := \diag(x)\).
For any \(\mu > 0\), the barrier Hessian is unbounded as any component of \(x\) approaches a bound, which is often where a solution is located.
Primal methods have long been superseded by the better-behaved primal-dual methods---see, e.g., \citep{forsgren-gill-wright-2022} and references therein for an overview of the extensive literature on the subject---in which the barrier Hessian is replaced with \(\nabla^2 \phi(x) + X^{-1} Z\), where \(Z := \diag(z)\) and \(z\) is an approximation of the vector of Lagrange multipliers for \(x \geq 0\).
Even though the primal-dual Hessian does not grow unbounded as fast as the primal Hessian, it nevertheless remains unbounded as any component of \(x\) approaches a bound.
In order to converge, interior-point methods rely on extra mechanisms that prevent components of \(x\) from approaching a bound too fast unless there are indications that a solution is nearby and \(\mu\) is close to zero.
In spite of those mechanisms, \(x\) must be allowed to approach bounds, and, therefore, the primal and primal-dual Hessians must be allowed to grow unbounded.
Although primal-dual interior-point methods can be shown to have excellent worst-case complexity bounds in convex optimization \citep{nesterov-nemirovskii-1994}, no such general result is known for nonconvex problems.

Another prime example, often cited in the literature, is when \(B_k\) results from a secant approximation \citep{dennis-more-1977}.
\citet[\S\(8.4\)]{conn-gould-toint-2000} suggest that for the BFGS and SR1 approximations, \(B_k\) could potentially grow by at most a constant at each update, though it is not clear whether that bound is attained.
This point is developed further in the related research below.

The paper is organized as follows.
\Cref{sec:context} provides the nonsmooth analysis background necessary to understand the algorithm of \citet{aravkin-baraldi-orban-2022}, a description of how models are constructed at each iteration, and a formal statement of the algorithm.
In \Cref{sec:ubnd-qn}, we establish convergence and a worst-case evaluation complexity bound under the assumption that the growth of the model Hessian is controlled by a function of the number of successful iterations, i.e., iterations in which a step is accepted.
We show in \Cref{sec:sharp} that the worst-case bound is indeed attained, by performing an analysis similar to that of \citep[Theorem~\(2.2.3\)]{cartis-gould-toint-2022}.
In \Cref{sec:num-verif}, we construct an explicit function that attains the bound and validate our findings numerically.
We provide concluding comments and perspectives in \Cref{sec:discussion}.

\subsection*{Related research}

We do not provide an extensive review of trust-region approaches for smooth optimization, but refer the interested reader to \citep{conn-gould-toint-2000} for a thorough account, as well as a number of generalizations.

We begin by reviewing milestones in the convergence analysis of trust-region methods with potentially unbounded model Hessians.
\citet{powell-1975} first showed convergence of a trust-region algorithm for smooth optimization that allows unbounded Hessian approximations \(B_k\).
Specifically, he assumes that there exist nonnegative \(\alpha\) and \(\beta\) such that \(\|B_k\| \leq \alpha + \beta \sum_{j=0}^{k-1} \|s_j\|\), where \(s_j\) is the trust-region step at iteration \(j\).
Under that and other standard assumptions, he established that \(\liminf \|\nabla f(x_k)\| = 0\).
\citeauthor{powell-1975} hints that his motivation lies in Hessian approximations arising from secant updates \citep{dennis-more-1977}.
To the best of our knowledge, it is not known whether secant approximations remain bounded.
However, \citet{fletcher-1972} establishes that the quasi-Newton update that bears Powell's name, the Powell symmetric Broyden update, derived in \citep{powell-1970}, satisfies the bound above.

Secant, and, in particular, quasi-Newton, methods are among the most widely employed methods in smooth optimization.
Yet, for lack of a boundedness result, no existing complexity analysis applies to them.
Like \citet{powell-1975}, our main motivation is to provide a first worst-case complexity result that may apply to them.
Whether or not certain quasi-Newton approximations satisfy our assumption on the growth of model Hessians remains to be established, even for convex problems.
Nevertheless, our result is a first step forward.

\citet{powell-1984} refines his earlier analysis by showing global convergence under the weaker assumption \(\|B_k\| \leq \alpha + \beta k\).
Under the weaker yet assumption
\begin{equation}
  \label{eq:asm-sum-invmaxBj-inf}
  \sum_{k=0}^{\infty} \frac{1}{1 + \max_{0 \leq j \leq k} \|B_j\|} = \infty,
\end{equation}
which is hinted at in the proofs of \citet{powell-1984}, \citet{toint-1988} shows that global convergence is preserved.
The condition is necessary but not sufficient; \citet{toint-1988} provides an example for which~\eqref{eq:asm-sum-invmaxBj-inf} fails to hold and on which trust-region methods may fail to converge.

When \(f\) is convex with uniformly bounded Hessian, \citet[\S\(8.4\)]{conn-gould-toint-2000} indicate that the BFGS update satisfies \(\|B_{k+1}\| \le \|B_k\| + \beta\) for some \(\beta \geq 0\).
Therefore, \(\|B_{k+1}\| \leq \|B_0\| + (k+1) \beta\), and the assumption of \citet{powell-1984}, and hence~\eqref{eq:asm-sum-invmaxBj-inf}, are satisfied.
The SR1 update with safeguards satisfies a similar inequality without the convexity assumption.

Under such a growth assumption, \citet{powell-2010} surmises in his concluding remarks that trust-region methods may require a ``monstrous'' number of iterations; which he projects to be exponential.

Because quasi-Newton approximations are typically only updated on successful iterations, i.e., when a trial step is accepted, we believe that the authors above mean that \(\|B_{k+1}\| \leq \|B_0\| + |\mathcal{S}_{k+1}| \beta\) instead, where \(|\mathcal{S}_{k+1}|\) is the number of successful iterations until iteration \(k+1\).
Our complexity result, though it does not encompass the latter bound, approaches it by imposing instead \(\|B_{k+1}\| \leq \|B_0\| + |\mathcal{S}_{k+1}|^p \beta\) for \(0 \leq p < 1\), and is therefore a first step towards validating \citeauthor{powell-2010}'s conjecture.

\citet{carter-1987} presents procedures to safeguard Hessian approximations in trust-region algorithms for smooth problems.
The goal of these procedures is to satisfy the \emph{uniform predicted decrease condition}
\begin{equation*}
  \varphi_k(x_k) - \varphi_k(x_{k+1}) \ge \tfrac{1}{2}\beta_1 \|\nabla f(x_k)\| \min \left(\Delta_k, \frac{\|\nabla f(x_k)\|}{\beta_0} \right),
\end{equation*}
where \(\varphi_k\) is a model of \(f\) about iterate \(x_k\), \(\Delta_k > 0\) is the trust-region radius, \(\beta_0 > 0\), and \(\beta_1 > 0\).
When \(\|B_k\| \le \beta_0\) for all \(k\), this condition is satisfied, but the author shows that it can also be satisfied under milder assumptions.
\citeauthor{carter-1987}’s procedures are used to correct \(B_k\) so that such assumptions hold.

We now review determinant complexity analyses of trust-region and related methods for smooth optimization.
\citet{cartis-gould-toint-2010} show that the steepest descent method and Newton's method for smooth problems may converge in as many as \(O(\epsilon^{-2})\) iterations, and that the bound is sharp for the steepest descent method.
The analysis assumes that the Hessian remains uniformly bounded.
In addition, they prove that it is possible to construct an example where Newton's method is arbitrarily slow when allowing unbounded Hessians.

Our main contribution is to establish that TR, the trust-region algorithm of \citep{aravkin-baraldi-orban-2022}, may converge in as many as \(O(\epsilon^{-2 / (1-p)})\) iterations, where \(p \in [0,1)\) is a parameter that controls the growth of the model Hessian---the larger \(p\), the larger the allowed growth.
Because \(\epsilon^{-2 / (1-p)} \to +\infty\) as \(p \nearrow 1\), our results reinforce that of \citet{cartis-gould-toint-2010} and makes it more precise.
Our analysis applies to smooth optimization---indeed, the example that we construct to establish sharpness of the complexity bound is smooth---but it is general enough to apply to~\eqref{eq:nlp}.

\citet[Section~\(2.2\)]{cartis-gould-toint-2022} show that the steepest-descent algorithm with backtracking Armijo linesearch results in an \(O(\epsilon^{-2})\) complexity bound, and a function is constructed by polynomial interpolation to prove that the bound is sharp, with a technique that is different from that of \citep{cartis-gould-toint-2010}.
The rest of their book reviews complexity analyses for trust-region and regularization methods, always under the assumption that the Hessian remains bounded.

The complexity of other methods for smooth optimization was subsequently analyzed using techniques similar to those of \citep{cartis-gould-toint-2010}.
The Adaptive Regularization with Cubics algorithm (ARC, or AR2 because it uses second-order derivatives) \citep{cartis-gould-toint-2011b,dussault-migot-orban-2023} minimizes at each iteration the model
\begin{equation}
  \label{eq:model-ar2}
  \varphi_k(x_k + s) = f(x_k) + \nabla f(x_k)^T s + \tfrac{1}{2}s^T B_k s + \tfrac{1}{3}\sigma_k \|s\|^3,
\end{equation}
where \(B_k\) must remain bounded.
It is known to require at most \(O(\epsilon^{-3/2})\) iterations to reach \(\|\nabla f(x_k)\| \le \epsilon\), and this bound is sharp \citep{nesterov-polyak-2006,cartis-gould-toint-2011b}.
\citet{curtis-robinson-samadi-2017} and \citet{martinez-raydan-2017} present modified trust-region algorithms with bounded model Hessians to solve nonconvex smooth problems that also have a complexity bound of \(O(\epsilon^{-3/2})\).


\citet{cartis-gould-toint-2020} show that Algorithm ARp for smooth problems, a generalization of ARC using a model of order \(p \ge 1\), requires at most \(O(\epsilon^{-(p + 1) / p})\) iterations to satisfy \(\|\nabla f(x_k)\| \le \epsilon\), and that the bound is sharp.
They introduce a generalization of the first-order stationarity measure \(\|\nabla f(x_k)\| \le \epsilon\) to \(q\)-th order stationarity, where \(q \in \N_0\), and show that at most \(O(\epsilon^{-(p + 1) / (p - q + 1)})\) evaluations of the objective and the derivatives are required with this measure.
They require that the \(p\)-th derivative of \(f\) be globally Hölder continuous.
For \(p = 2\) and \(q = 1\), we recover the bound of \citep{cartis-gould-toint-2011b}.

For smooth nonconvex problems with bounded Hessians, the number of iterations required to satisfy the conditions on the gradient \(\|\nabla f(x_k)\| \le \epsilon_g\) and on the smallest eigenvalue of the Hessian \(\lambda_{\min}(\nabla^2 f(x_k)) \ge -\epsilon_H\), where \(\epsilon_g\), \(\epsilon_H \in (0,1)\), have also been studied.
\citet{cartis-gould-toint-2012} show that their trust-region algorithm needs at most \(O(\max \{\epsilon_g^{-2}\epsilon_H^{-1}, \epsilon_H^{-3}\})\) iterations to satisfy these conditions, and \(O(\max \{\epsilon_g^{-3/2}, \epsilon_H^{-3}\})\) iterations for ARC\@.
The latter bound is also obtained for the trust-region algorithms in \citep{curtis-robinson-samadi-2017,martinez-raydan-2017}.
\citet{royer-wright-2018} use a second-order linesearch method to obtain the bound \(O(\max\{\epsilon_g^{-3}\epsilon_H^3, \epsilon_g^{-3/2}, \epsilon_H^{-3}\})\).

\citet{aravkin-baraldi-orban-2022} provide an overview of the literature on convergence of methods for nonsmooth optimization, and we now summarize the review with an eye to trust-region methods.
Methods prior to their work were restricted to special cases.
Most were developed for \(f = 0\), i.e., in a purely nonsmooth context.
\citet{yuan-1985} considers a nonsmooth term of the form \(h(c(x))\), where \(c \in \mathcal{C}^1\) and convex.
\citet{dennis-li-tapia-1995} take \(f = 0\) and assume that \(h\) is Lipschitz-continuous.
\citet{qi-sun-1994} relax the assumptions of \citep{dennis-li-tapia-1995} to \(h\) locally Lipschitz-continuous with bounded level sets.
\citet{martinez-moretti-1997} add treatment of equality constraints to the method of \citet{qi-sun-1994}.
The only prior trust-region method for \(f \neq 0\) and more general \(h\) that we are aware of is that of \citet{kim-sra-dhillon-2010}, who assume that \(f\) and \(h\) are convex.
None of those works provides a complexity analysis.

Finally, we review complexity analyses of trust-region methods for nonsmooth problems.
\citet{cartis-gould-toint-2011c} describe a first-order trust-region method and a quadratic regularization algorithm to solve nonsmooth problems of the form
\begin{equation}
  \label{eq:nlp-composite}
  \minimize{x \in \R^n} \ f(x) + h(c(x)),
\end{equation}
where \(f\) and \(c\) are continuously differentiable and may be nonconvex, and \(h\) is convex but may be nonsmooth, and is Lipschitz-continuous.
Note that~\eqref{eq:nlp-composite} is a special case of~\eqref{eq:nlp}, but the convexity assumption on \(h\) is strong.
They show that both algorithms have a complexity bound of \(O(\epsilon^{-2})\).
\citet{grapiglia-yuan-yuan-2016} provide a unified convergence theory for smooth optimization that has trust-region methods as a special case.
They also generalize the results of \citep{cartis-gould-toint-2011c} under the same assumptions.

\citet{aravkin-baraldi-orban-2022} describe a proximal trust-region algorithm to solve~\eqref{eq:nlp} using bounded model Hessians.
They also present a quadratic regularization variant.
They establish that their criticality measure is smaller than \(\epsilon\) in at most \(O(\epsilon^{-2})\) iterations for both algorithms.
\citet{aravkin-baraldi-orban-2022b} adapt these algorithms to solve nonsmooth regularized least-squares problems and obtain the same complexity bound under the assumption that the residual Jacobian is uniformly bounded.
As far as we know, the complexity analyses of \citep{aravkin-baraldi-orban-2022,aravkin-baraldi-orban-2022b} make the weakest assumptions on \(h\) so far, that \(h\) be lsc.

\citet{baraldi-kouri-2022} also describe a proximal trust-region algorithm for convex \(h\).
In addition, they allow the use of inexact objective and gradient evaluations.
As \citet{toint-1988} in the smooth case, they assume that
\begin{equation}
  \label{eq:asm-sum-invmaxomegaj-inf}
  \sum_{k=0}^{\infty} \frac{1}{1 + \max_{0 \leq j \leq k} \omega_j} = \infty,
\end{equation}
where
\[
  \omega_k = \sup \left\{\frac{2}{\|s\|^2} | \varphi_k(x_k + s) - \varphi_k(x_k) - \nabla \varphi_k(x_k)^T s| \mid 0 < \|s\| \le \Delta_k \right\},
\]
and \(\varphi_k\) is a smooth model of \(f\) about \(x_k\).
In particular, if \(\varphi_k\) is a second-order Taylor approximation at \(x_k\) with Hessian approximation \(B_k\), \(\omega_k = \sup \left\{s^T B_k s / \|s\|^2 \mid 0 < \|s\| \le \Delta_k \right\}\), so that~\eqref{eq:asm-sum-invmaxomegaj-inf} is reminiscent of~\eqref{eq:asm-sum-invmaxBj-inf}.
If \(\omega_k\) is bounded independently of \(k\), which is the case for bounded Hessian approximations, they show that their algorithm enjoys a complexity bound of \(O(\epsilon^{-2})\).

\citet{cartis-gould-toint-2020b} present a similar concept of high-order approximate minimizers to that of \citep{cartis-gould-toint-2020} for nonsmooth problems such as~\eqref{eq:nlp-composite} where \(f\), \(c\) are smooth, and \(h\) is nonsmooth but Lipschitz-continuous.
They present an algorithm of adaptive regularization of order \(p\), and derive several bounds depending on the properties of~\eqref{eq:nlp-composite} and of the order of the desired approximate minimizer.
In particular, for \(q = 1\) and convex \(h\), their complexity bound is \(O(\epsilon^{-(p+1) / p})\), and they show that it is sharp.

\subsection*{Contributions}

Our main contribution is a sharp \(O(\epsilon^{-2/(1-p)})\) worst-case evaluation complexity bound for a class of trust-region algorithms for smooth and nonsmooth optimization when model Hessians \(B_k\) are allowed to grow according to \(\|B_k\| = O(|\mathcal{S}_k|^p|)\), where \(|\mathcal{S}_k|\) is the number of successful iterations up to iteration \(k\), and \(0 \leq p < 1\).
Our analysis builds upon the intuition of \citet{powell-2010} and Hermite interpolation-based tools inspired from those of \citet{cartis-gould-toint-2022}.
The trust-region algorithm, \Cref{alg:tr-nonsmooth}, is a minor variation on that of \citet{aravkin-baraldi-orban-2022} to allow for potentially unbounded model Hessians.
To the best of our knowledge, previous literature does not provide a complexity analysis in the case of potentially unbounded model Hessians.
Our result applies to nonconvex nonsmooth regularized optimization problems of the form~\eqref{eq:nlp}, and to smooth optimization as a special case.
Indeed, the example constructed in \Cref{sec:sharp} to establish sharpness is for smooth optimization, i.e., \(h = 0\).
Finally, we provide new results that indicate conditions under which limit points of the sequence of iterates are stationary.

\subsection*{Notation}

\(\B\) denotes the closed unit ball at the origin in a certain norm dictated by the context, \(\Delta \B\) is the ball of radius \(\Delta > 0\) centered at the origin, and \(x + \Delta \B\) is the ball of radius \(\Delta > 0\) centered at \(x \in \R^n\).
For \(A \subseteq \R^n\), the indicator of \(A\) is \(\chi(\cdot \mid A): \R^n \to \R \cup \{+\infty\}\) defined as \(\chi(x \mid A) = 0\) if \(x \in A\) and \(+\infty\) otherwise.
If \(A \neq \varnothing\), \(\chi(\cdot \mid A)\) is proper.
If \(A\) is closed, \(\chi(\cdot \mid A)\) is lsc.
For a finite set \(A \subset \N\), we denote by \(|A|\) its cardinality.
If \(f_1\) and \(f_2\) are two positive functions of \(\epsilon > 0\), we say that \(f_1(\epsilon) = O(f_2(\epsilon))\) if there exists a constant \(C > 0\) such that \(f_1(\epsilon) \leq C f_2(\epsilon)\) for all \(\epsilon > 0\) sufficiently small.
\(\|\cdot\|\) denotes the \(2\)-norm on \(\R^n\), and its associated induced matrix spectral norm on \(\R^{n \times n}\) is also denoted \(\|\cdot\|\).

\section{Context}%
\label{sec:context}

\subsection{Background}

We recall relevant concepts of variational analysis---see, e.g., \citep{rockafellar-wets-1998}.

Consider \(\phi: \R^n \to \widebar{\R}\) and \(\bar{x} \in \R^n\) with \(\phi(\bar{x}) < \infty\).
The \emph{Fr\'echet subdifferential} of \(\phi\) at \(\bar{x}\) is the closed convex set \(\widehat{\partial} \phi(\bar{x})\) of \(v \in \R^n\) such that
\[
  \liminf_{\substack{x \to \bar{x} \\
      x \neq \bar{x}}} \frac{\phi(x) - \phi(\bar{x}) - v^T (x - \bar{x})}{\|x - \bar{x}\|} \geq 0.
\]

The \emph{limiting subdifferential} of \(\phi\) at \(\bar{x}\) is the closed, but not necessarily convex, set \(\partial \phi(\bar{x})\) of \(v \in \R^n\) for which there exist \(\{x_k\} \to \bar{x}\) and \(\{v_k\} \to v\) such that \(\{\phi(x_k)\} \to \phi(\bar{x})\) and \(v_k \in \widehat{\partial} \phi(x_k)\) for all \(k\).
\(\widehat{\partial} \phi(\bar{x}) \subset \partial \phi(\bar{x})\) always holds.

We say that \(\bar{x}\) is \emph{stationary} for the problem of minimizing \(\phi\) if \(0 \in \partial \phi(\bar{x})\).

The \emph{horizon subdifferential} of \(\phi\) at \(\bar{x}\) is the closed, but not necessarily convex, cone \(\partial^{\infty} \phi(\bar{x})\) of \(v \in \R^n\) for which there exist \(\{x_k\} \to \bar{x}\), \(\{v_k\}\) and \(\{\lambda_k\} \downarrow 0\) such that \(\{\phi(x_k)\} \to \phi(\bar{x})\), \(v_k \in \widehat{\partial} \phi(x_k)\) for all \(k\), and \(\{\lambda_k v_k\} \to v\).

If \(C \subseteq \R^n\) and \(\bar{x} \in C\), the closed convex cone \(\widehat{N}_C(\bar{x}) := \widehat{\partial} \chi(\bar{x} \mid C)\) is the regular normal cone to \(C\) at \(\bar{x}\).
The closed cone \(N_C(\bar{x}) := \partial \chi(\bar{x} \mid C) = \partial^{\infty} \chi(\bar{x} \mid C)\) is the normal cone to \(C\) at \(\bar{x}\).
\(\widehat{N}_C(\bar{x}) \subseteq N_C(\bar{x})\) always holds, and is an equality if \(C\) is convex.

\(\phi\) is \emph{proper} if \(\phi(x) > - \infty\) for all \(x\), and \(\phi(x) < \infty\) for at least one \(x\).
\(\phi\) is \emph{lower semicontinuous} (\emph{lsc}) at \(\bar{x}\) if \(\liminf_{x \rightarrow \bar{x}} \phi(x) = \phi(\bar{x})\).

Let \(\phi: \R^n \to \widebar{\R}\) be proper lsc, and \(C \subseteq \R^n\) be closed.
We say that the \emph{constraint qualification} is satisfied at \(\bar{x} \in C\) for the constrained problem
\begin{equation}%
  \label{eq:min-phi-constrained}
  \minimize{x \in \R^n} \ \phi(x) \quad \st \ x \in C
\end{equation}
if
\begin{equation}%
  \label{eq:cq}
  \partial^{\infty} \phi(\bar{x}) \cap (-N_C(\bar{x})) = \{0\}.
\end{equation}

If \(\bar{x}\) solves~\eqref{eq:min-phi-constrained} and~\eqref{eq:cq} is satisfied at \(\bar{x}\), \citep[Theorem~\(8.15\) and Corollary~\(10.9\)]{rockafellar-wets-1998} yield
\[
  0 \in \partial (\phi + \chi(\cdot \mid C))(\bar{x}) \subseteq \partial \phi(\bar{x}) + N_C(\bar{x}).
\]
In the case of~\eqref{eq:nlp}, this first-order necessary condition for optimality reads
\[
  0 \in \nabla f(\bar{x}) + \partial h(\bar{x}) + N_{[\ell, u]}(\bar{x})
\]
thanks to \citep[Exercise~\(8.8c\)]{rockafellar-wets-1998}.

If \(\phi_k\) and \(\phi : \R^n \to \widebar{\R}\) for \(k \in \N\), we say that \(\{\phi_k\}\) converges to \(\phi\) \emph{continuously} if \(\{\phi_k(x_k)\} \to \phi(x)\) for all sequences \(\{x_k\} \to x\) in \(\R^n\).

The \emph{epigraph} of \(\phi\) is the set \(\epi \phi \coloneq \{ (t, x) \mid t \geq \phi(x) \} \subseteq \R \times \R^n\).
The set \(\epi \phi\) is closed if and only if \(\phi\) is lsc.

For a sequence of sets \(\{A_k\}\) with \(A_k \subseteq \R^n\) for all \(k \in \N\), the set \(\limsup_{k \in \N} A_k\) is the set of limits of all possible subsequences \(\{x_k\}_N\) with \(N \subseteq \N\) infinite and \(x_k \in A_k\) for all \(k \in N\).
The set \(\liminf_{k \in \N} A_k\) is the set of limits of sequences \(\{x_k\}_{k \in \N}\) such that \(x_k \in A_k\) for all \(k\) sufficiently large.
In particular, those concepts can be applied to the sets \(\epi \phi_k\) where \(\phi_k: \R^n \to \R\) for \(k \in \N\).
The sets \(\liminf_k \epi \phi_k\) and \(\limsup_k \epi \phi_k\) enjoy the properties of epigraphs, i.e., if \((t, x)\) lies in one of them, so does \((s, x)\) for all \(s \geq t\).
In addition, both are closed, and therefore, can be viewed as the epigraphs of certain lsc functions.
The lower and upper epi-limits of \(\{\phi_k\}\) are the functions \(\eliminf_k \phi_k\) and \(\elimsup_k \phi_k\) that satisfy \(\epi \eliminf_k \phi_k = \limsup_k \epi \phi_k\) and \(\epi \elimsup_k \phi_k = \liminf_k \epi \phi_k\).
In general, \(\eliminf_k \phi_k \leq \elimsup_k \phi_k\).
When they coincide, we say that \(\{\phi_k\}\) converges epigraphically to the common value \(\phi\), and write \(\{\phi_k\} \eto \phi\) or \(\elim_k \phi_k = \phi\).

The \emph{proximal operator} associated with a proper lsc function \(\phi\) is
\begin{equation}%
  \label{eq:def-prox}
  \prox_{\nu \phi}(q) := \argmin{x} \ \tfrac{1}{2} \nu^{-1} \|x - q\|_2^2 + \phi(x),
\end{equation}
where \(\nu > 0\) is a preset steplength.
Below, we assume that all proximal operators can be evaluated analytically.
That is not a restrictive assumption in many cases of interest for applications---see \citep{beck-2017} for a large, but not exhaustive, list of choices of \(\phi\) for which the set~\eqref{eq:def-prox} is known.

We say that \(\phi\) is prox-bounded if it is bounded below by a quadratic.
If \(\phi\) is prox-bounded and \(\nu > 0\) is sufficiently small, \(\prox_{\nu \phi}(q)\) is a nonempty and closed set.
It may contain multiple elements.

The proximal gradient method \citep{lions-mercier-1979,fukushima-mine-1981} for~\eqref{eq:nlp} is a generalization of the gradient method that takes the nonsmooth term into account.
It generates iterates \(\{s_j\}\) according to
\begin{equation}%
  \label{eq:pg-iter-basic}
  s_{j+1} \in \prox_{\nu h}(s_j - \nu \nabla f(s_j)).
\end{equation}

\subsection{Models and trust-region algorithm}

At \(x \in \R^n\) where \(h\) is finite, we define models
\begin{subequations}%
  \label{eq:def-models}
  \begin{align}
    \varphi(s; x) & \phantom{:}\approx f(x + s)
    \label{eq:def-varphi}%
    \\
    \psi(s; x)    & \phantom{:}\approx h(x + s)
    \label{eq:def-psi}%
    \\
    m(s; x)       & := \varphi(s; x) + \psi(s; x).
    \label{eq:def-m}
  \end{align}
\end{subequations}
Our assumptions on~\eqref{eq:def-models} are a minor variation of those of \citet{aravkin-baraldi-orban-2022}:

\begin{modelassumption}%
  \label{asm:models}
  For any \(x \in \R^n\), \(\varphi(\cdot; x) \in \mathcal{C}^1\), and satisfies \(\varphi(0; x) = f(x)\) and \(\nabla \varphi(0; x) = \nabla f(x)\).
  For any \(x \in \R^n\) where \(h\) is finite, \(\psi(\cdot; x)\) is proper lsc, and satisfies \(\psi(0; x) = h(x)\) and \(\partial \psi(0; x) \subseteq \partial h(x)\).
\end{modelassumption}

The difference between \Cref{asm:models} and \citep[Model Assumption~\(3.1\)]{aravkin-baraldi-orban-2022} is the last inclusion instead of an equality between the subdifferentials.

The following result states that if \(s = 0\) minimizes~\eqref{eq:def-m} and~\eqref{eq:cq} is satisfied, \(x\) must be stationary.

\begin{proposition}[{\protect \citealp[Proposition~\(1\)]{leconte-orban-2024}}]%
  \label{prop:iprox-stationarity}
  Let \(C \subset \R^n\) be nonempty and compact, and let \Cref{asm:models} be satisfied.
  Let~\eqref{eq:nlp} satisfy the constraint qualification~\eqref{eq:cq} at \(x \in C\).
  Assume \(0 \in \argmin{s} m(s; x) + \chi(x + s \mid C)\), and let the latter subproblem satisfy the constraint qualification~\eqref{eq:cq} at \(s = 0\).
  Then \(x\) is first-order stationary for~\eqref{eq:nlp}.
\end{proposition}

Assuming \(\partial \psi(0; x) = \partial h(x)\) in \Cref{asm:models} would allow us to establish the reverse implication in \Cref{prop:iprox-stationarity}.

Each iteration is divided into two parts.
In the first part, \citet{aravkin-baraldi-leconte-orban-2023} define the following model based on a first-order Taylor expansion to compute a \emph{Cauchy point}
\begin{subequations}%
  \label{eq:def-model-nu}
  \begin{align}
    \varphi_{\mathrm{cp}}(s; x) & := f(x) + \nabla f(x)^T s,
    \label{eq:def-phi-nu}%
    \\
    m_{\mathrm{cp}}(s; x, \nu)                & := \varphi_{\mathrm{cp}}(s; x) + \tfrac{1}{2}\nu^{-1}\|s\|^2 + \psi(s; x),
    \label{eq:def-m-nu}
  \end{align}
\end{subequations}
where \(\nu_k > 0\) and ``cp'' stands for ``Cauchy point.''
We compute a first step
\begin{equation}%
  \label{eq:ipg-iter-sk1}
  s_{k,1} \in \argmin{s} \ m_{\mathrm{cp}}(s; x_k, \nu_k) + \chi(x_k + s \mid [\ell, \, u] \cap (x_k + \Delta_k \B)),
\end{equation}
for an appropriate value of \(\nu_k > 0\).

In the notation of \citep{aravkin-baraldi-leconte-orban-2023}, let
\begin{equation}
  \label{eq:def-xi1}
  \xi_{\mathrm{cp}}(\Delta_k; x_k, \nu_k) := f(x_k) + h(x_k) - \varphi_{\mathrm{cp}}(s_{k,1}; x_k) - \psi(s_{k,1}; x_k),
\end{equation}
denote the optimal model decrease for~\eqref{eq:def-model-nu}.
By~\eqref{eq:ipg-iter-sk1},
\begin{equation*}
  m_{\mathrm{cp}}(s_{k,1}; x_k, \nu_k) = \varphi_{\mathrm{cp}}(s_{k,1}; x_k) + \psi(s_{k,1}; x_k) + \tfrac{1}{2} \nu_k^{-1} \|s_{k,1}\|^2 \le m_{\mathrm{cp}}(0; x_k, \nu_k) = f(x_k) + h(x_k),
\end{equation*}
so that, with~\eqref{eq:def-xi1},
\begin{equation}%
  \label{eq:xi1-suff-decrease}
  \xi_{\mathrm{cp}}(\Delta_k; x_k, \nu_k) \ge \tfrac{1}{2}\nu_k^{-1} \|s_{k,1}\|^2.
\end{equation}

The following proposition indicates that \(\xi_{\mathrm{cp}}(\Delta; x, \nu)\) can be used to determine whether \(x\) is first-order stationary for~\eqref{eq:nlp}.

\begin{proposition}[{\protect \citealp[Proposition~\(3.3\)]{aravkin-baraldi-orban-2022}} and \citealp{aravkin-baraldi-leconte-orban-2023}]%
  \label{prop:xicp-stat-0}
  Let \Cref{asm:models} be satisfied, \(\Delta > 0\), and \(\nu > 0\).
  In addition, let~\eqref{eq:nlp} satisfy the constraint qualification at \(x\) and the objective of~\eqref{eq:ipg-iter-sk1} satisfy the constraint qualification at \(s = 0\).
  Then, \(\xi_{\mathrm{cp}}(\Delta; x, \nu) = 0\) \(\iff\) \(s = 0\) is a solution of~\eqref{eq:ipg-iter-sk1} \(\implies\) \(x\) is first-order stationary for~\eqref{eq:nlp}.
\end{proposition}

In the second part of iteration \(k\), we construct a model based on the second-order Taylor expansion
\begin{subequations}%
  \label{eq:def-model-quad}
  \begin{align}
    \varphi(s; x, B) & := f(x) + \nabla f(x)^T s + \tfrac{1}{2} s^T B s,
    \label{eq:def-phi-quad}%
    \\
    m(s; x, B)       & := \varphi(s; x, B) + \psi(s; x),
  \end{align}
\end{subequations}
where \(B = B^T \in \R^{n \times n}\), and compute a step as an approximate solution of
\begin{equation}
  \label{eq:model-k}
  \minimize{s} \ m(s; x_k, B_k) + \chi(x_k + s \mid [\ell, \, u] \cap (x_k + \Delta_k \B)),
\end{equation}
using \(s_{k,1}\) as starting point.

We focus on the trust-region (TR) algorithm formally stated as \Cref{alg:tr-nonsmooth}.
It consists of the algorithm of \citet{aravkin-baraldi-orban-2022} with a modified maximum allowable stepsize \(\nu_k\).
The concept of inexact solution of~\eqref{eq:model-k} at \Cref{step:sk-computation} is made precise in \Cref{prop:justif-asm-model-diff} below.

Note that \Cref{alg:tr-nonsmooth} differs from a ``standard'' trust-region algorithm in that the parameter \(\Delta_k\) that is updated according to whether or not a step \(s_k\) is accepted serves to define the trust-region radius, but is not the radius in itself; \(s_{k,1}\) is used to check for stationarity, and to set the trust-region radius for the computation of the step \(s_k\).
\citet{aravkin-baraldi-orban-2022} provide more details on this point and link to variants of standard trust-region algorithms for smooth optimization possessing similar features.

\begin{algorithm}[ht]%
  \caption[caption]{%
    Nonsmooth trust-region algorithm with potentially unbounded Hessian.%
    \label{alg:tr-nonsmooth}
  }
  \begin{algorithmic}[1]%
    \State Choose constants
    \[
      0 < \eta_1 \leq \eta_2 < 1,
      \quad
      0 < 1/\gamma_3 \leq \gamma_1 \leq \gamma_2 < 1 < \gamma_3 \leq \gamma_4, \quad
      \Delta_{\max} > \Delta_0, \quad
      \alpha > 0,
      \quad \text{and} \quad
      \beta \geq 1.
    \]
    \State Choose a stopping tolerance \(\epsilon > 0\).
    \State Choose \(x_0 \in \R^n\) where \(h\) is finite, \(\Delta_0 > 0\), compute \(f(x_0) + h(x_0)\).
    \For{\(k = 0, 1, \ldots\)}
    \State%
    \label{step:nuk-interval}%
    Choose
    \begin{equation}
      \label{eq:nu-ineq}
      0 < \nu_k \le \frac{\alpha \Delta_k}{1 + \|B_k\|(1 + \alpha \Delta_k)} = \frac{1}{\alpha^{-1}\Delta_k^{-1} + \|B_k\|(1 + \alpha^{-1} \Delta_k^{-1})}.
    \end{equation}
    \State%
    \label{step:sk1-computation}%
    Define \(m_{\mathrm{cp}}(s; x_k, \nu_k)\) as in~\eqref{eq:def-model-nu} and compute \(s_{k,1}\) as in~\eqref{eq:ipg-iter-sk1}.
    \State%
    \label{step:stop-crit}
    If \(\nu_k^{-1/2} \xi_{\mathrm{cp}}(\Delta_k; x_k, \nu_k)^{1/2} \leq \epsilon\), terminate and claim that \(x_k\) is approximately stationary.
    \State%
    \label{step:sk-computation}%
    Define \(m(s; x_k, B_k)\) as in~\eqref{eq:def-model-quad} according to \Cref{asm:models} and compute an approximate solution \(s_k\) of~\eqref{eq:model-k} with \(\Delta_k\) replaced by \(\min (\Delta_k, \, \beta \|s_{k,1}\|)\).
    \State%
    \label{step:rhok-computation}%
    Compute the ratio
    \begin{equation}%
      \label{eq:rhok}
      \rho_k :=
      \frac{%
        f(x_k) + h(x_k) - (f(x_k + s_k) + h(x_k + s_k))
      }{
        m(0; x_k, B_k) - m(s_k; x_k, B_k)
      }.
    \end{equation}
    \State%
    \label{setp:xk-update}%
    If \(\rho_k \geq \eta_1\), set \(x_{k+1} = x_k + s_k\).
    Otherwise, set \(x_{k+1} = x_k\).
    \State%
    \label{step:deltak-update}%
    Update the trust-region radius according to
    \[
      \bar \Delta_{k+1} \in
      \left\{
      \begin{array}{lll}%
        {[\gamma_3 \Delta_k, \, \gamma_4 \Delta_k]} &
        \text{ if } \rho_k \geq \eta_2,             &
        \text{(very successful iteration)}
        \\
        {[\gamma_2 \Delta_k, \, \Delta_k]}          &
        \text{ if } \eta_1 \leq \rho_k < \eta_2,    &
        \text{(successful iteration)}
        \\
        {[\gamma_1 \Delta_k, \, \gamma_2 \Delta_k]} &
        \text{ if } \rho_k < \eta_1,                &
        \text{(unsuccessful iteration)}
      \end{array}
      \right.
    \]
    and \(\Delta_{k+1} = \min(\bar \Delta_{k+1}, \, \Delta_{\max})\)
    \EndFor
  \end{algorithmic}
\end{algorithm}

Let us now briefly turn our attention to unconstrained smooth problems.
In this case, the following lemma gives a global minimizer of~\eqref{eq:def-model-nu} and~\eqref{eq:def-model-quad}.

\begin{lemma}%
  \label{lem:sol-pb-B-gen}
  We consider the special case of~\eqref{eq:nlp} where \(h = 0\), \(\ell_i = -\infty\) and \(u_i = +\infty\) for \(i = 1, \ldots , n\).
  Let \(B = B^T \in \R^{n \times n}\) be positive definite and \(\psi = 0\).
  Then for any \(x \in \R^n\),
  \begin{equation}
    \label{eq:pb-B-gen}
    \argmin{s} m(s; x, B) = \argmin{s} \varphi(s; x, B) = \{-B^{-1} \nabla f(x)\}.
  \end{equation}
  In particular, if \(B = \nu^{-1} I\) with \(\nu > 0\),
  \begin{equation}
    \label{eq:pb-nu-cp-gen}
    \argmin{s} m_{\mathrm{cp}}(s; x, \nu) = \argmin{s} \varphi_{\mathrm{cp}}(s; x) + \tfrac{1}{2} \nu^{-1} \|s\|^2 = \{s_{k,1}\} = \{-\nu \nabla f(x)\}.
  \end{equation}
\end{lemma}

\begin{proof}
  The objective of~\eqref{eq:pb-B-gen} is convex because \(B\) is positive definite.
  Its global minimizer satisfies the first-order necessary condition \(\nabla f(x) + B s = 0\), i.e., \(s = -B^{-1} \nabla f(x)\).
  With \(B = \nu^{-1} I\), the first-order necessary condition is \(s = -\nu \nabla f(x)\).
\end{proof}

The following proposition draws a parallel between \(\xi_{\mathrm{cp}}(\Delta_k; x_k, \nu_k)\) and \(\|\nabla f(x_k)\|\) for smooth problems when the trust-region constraint is inactive, as is expected to occur when close to a stationary point.

\begin{proposition}%
  \label{prop:smooth-stat}
  We consider the special case of~\eqref{eq:nlp} where \(h = 0\), \(\ell_i = -\infty\) and \(u_i = +\infty\) for \(i = 1, \ldots , n\).
  If \(\|s_{k,1}\| < \Delta_k\), then \(\xi_{\mathrm{cp}}(\Delta_k; x_k, \nu_k) = \nu_k \|\nabla f(x_k)\|^2\).
\end{proposition}

\begin{proof}
  If the trust-region constraint is inactive, \Cref{lem:sol-pb-B-gen} indicates that \(s_{k,1} = -\nu_k \nabla f(x_k)\).
  Thus,~\eqref{eq:def-xi1} yields \(\xi_{\mathrm{cp}}(\Delta_k; x_k, \nu_k) = -\nabla f(x_k)^T s_{k,1} = \nu_k\|\nabla f(x_k)\|^2\).
\end{proof}

\section{Convergence and complexity with potentially unbounded Hessians}%
\label{sec:ubnd-qn}

From this section onwards, we consider the model defined in~\eqref{eq:def-model-quad}, and we aim to establish convergence and worst-case complexity results for \Cref{alg:tr-nonsmooth} in the presence of potentially unbounded Hessian approximations \(B_k\).

The following two assumptions are essential.
\Cref{asm:model-diff} is \citep[Step Assumption~\(3.8b\)]{aravkin-baraldi-orban-2022}, whereas \Cref{asm:model-ubd} is a relaxed version of \citep[Step Assumption~\(3.8a\)]{aravkin-baraldi-orban-2022} that takes into account potentially unbounded Hessian approximations.
Indeed, assuming, for simplicity, that \(\nabla^2 f(x_k)\) exists, a second-order Taylor expansion of \(f\) about \(x_k\) yields
\begin{equation*}
  f(x_k + s_k) - \varphi(s_k; x_k, B_k) = \tfrac{1}{2}s_k^T (\nabla^2 f(x_k) - B_k) s_k + o(\|s_k\|^2),
\end{equation*}
which is not necessarily \(O(\|s_k\|^2)\) if \(\{B_k\}\) is unbounded.

\begin{assumption}%
  \label{asm:model-diff}
  There exists \(\kappa_{\mathrm{mdc}} \in (0, 1)\) such that
  \begin{equation}
    m(0; x_k, B_k) - m(s_k; x_k, B_k) \ge \kappa_{\mathrm{mdc}} \xi_{\mathrm{cp}}(\Delta_k; x_k, \nu_k).
  \end{equation}
\end{assumption}

\begin{assumption}%
  \label{asm:model-ubd}
  There exists \(\kappa_{\mathrm{ubd}} > 0\) such that
  \begin{equation}
    |(f + h)(x_k + s_k) - m(s_k; x_k, B_k)| \le \kappa_{\mathrm{ubd}} (1 + \|B_k\|)\|s_k\|_2^2.
  \end{equation}
\end{assumption}


\citet[Proposition~\(2\)]{leconte-orban-2024} and \citet{aravkin-baraldi-leconte-orban-2023} already indicate that \Cref{asm:model-diff} holds for TRDH and TR\@.
We now justify that it also holds for \Cref{alg:tr-nonsmooth} with potentially unbounded Hessian approximations.

\begin{proposition}%
  \label{prop:justif-asm-model-diff}
  If \Cref{asm:models} is satisfied, and \(s_k\) is computed so that \(m(s_k; x_k, B_k) \leq m(s_{k,1}; x_k, B_k)\) at \Cref{step:sk-computation} of \Cref{alg:tr-nonsmooth}, there exists \(\kappa_{\mathrm{mdc}} \in (0, 1)\) such that \Cref{asm:model-diff} holds.
\end{proposition}

\begin{proof}
  We proceed similarly as in \citep[Proposition~\(2\)]{leconte-orban-2024}.
  Note that \(s_{k,1}\) is feasible for the problem on \Cref{step:sk-computation}.
  The definition of \(s_k\) in the assumptions implies that
  \begin{align*}
    m(s_k; x_k, B_k) & \le m(s_{k,1}; x_k, B_k) = \varphi_{\mathrm{cp}}(s_{k,1}; x_k) + \tfrac{1}{2} s_{k,1}^T B_k s_{k,1} + \psi(s_{k,1}; x_k)
    \\
                     & \le \varphi_{\mathrm{cp}}(s_{k,1}; x_k) + \tfrac{1}{2} \|B_k\| \|s_{k,1}\|^2 + \psi(s_{k,1}; x_k),
  \end{align*}
  where we used Cauchy-Schwarz and the consistency of the \(\ell_2\)-norm for matrices.
  Because \(m(0; x_k, B_k) = m_{\mathrm{cp}}(0; x_k, \nu_k)\),
  \begin{equation*}
    m(0; x_k, B_k) - m(s_k; x_k, B_k) \ge \xi_{\mathrm{cp}}(\Delta_k; x_k, \nu_k) - \tfrac{1}{2} \|B_k\| \|s_{k,1}\|^2.
  \end{equation*}
  To satisfy \Cref{asm:model-diff}, it is sufficient to show that there exists \(\kappa_{\mathrm{mdc}} \in (0, 1)\) such that
  \begin{equation*}
    \xi_{\mathrm{cp}}(\Delta_k; x_k, \nu_k) - \tfrac{1}{2} \|B_k\| \|s_{k,1}\|^2 \ge \kappa_{\mathrm{mdc}} \xi_{\mathrm{cp}}(\Delta_k; x_k, \nu_k),
  \end{equation*}
  i.e.,
  \begin{equation*}
    (1 - \kappa_{\mathrm{mdc}}) \xi_{\mathrm{cp}}(\Delta_k; x_k, \nu_k) \ge \tfrac{1}{2} \|B_k\| \|s_{k,1}\|^2.
  \end{equation*}
  Because of~\eqref{eq:xi1-suff-decrease}, it is also sufficient to show that there exists \(\kappa_{\mathrm{mdc}} \in (0, 1)\) such that
  \begin{equation}
    \label{eq:kappa-mdc-nuk-Bk-ineq}
    (1 - \kappa_{\mathrm{mdc}%
    }%
    )
    \nu_k^{-1} \ge \|B_k\|.
  \end{equation}
  If \(B_k = 0\), the conclusion holds.
  Otherwise,
  \begin{equation}
    \label{eq:Bknuk-ineq}
    \|B_k\| \nu_k \le \frac{1}{\alpha^{-1} \Delta_k^{-1} \|B_k\|^{-1} + 1 + \alpha^{-1} \Delta_k^{-1}} \le \frac{1}{\alpha^{-1} \Delta_{\max}^{-1} \|B_k\|^{-1} + 1 + \alpha^{-1} \Delta_{\max}^{-1}} \le \frac{1}{1 + \alpha^{-1} \Delta_{\max}^{-1}} \in (0, 1).
  \end{equation}
  We deduce from~\eqref{eq:Bknuk-ineq} that~\eqref{eq:kappa-mdc-nuk-Bk-ineq} holds, which is sufficient to satisfy \Cref{asm:model-diff}.
\end{proof}

Because a step \(s_k\) is typically computed using a variant of the proximal gradient method applied to \(m(s; x_k, B_k)\), \Cref{prop:justif-asm-model-diff} suggests that we first compute \(s_{k,1}\) to determine (approximate) stationarity, and continue the proximal gradient iterations from \(s_{k,1}\) if appropriate.

Because all norms are equivalent in finite dimension, the proof of \Cref{prop:justif-asm-model-diff} continues to hold if we compute \(\|B_k\|\) in a norm other than the spectral norm, or even if we obtain an approximation \(\beta_k \geq \mu \|B_k\|\) for some \(\mu \in (0, 1)\).
We may replace \(\|B_k\|\) with \(\beta_k\) in the upper bound on \(\nu_k\) in \Cref{alg:tr-nonsmooth} and repeat the proof of \Cref{prop:justif-asm-model-diff} to arrive at \(\kappa_{\mathrm{mdc}} = 1 - 1 / (\mu (1 + \alpha^{-1} \Delta_{\max}^{-1}))\).
In practice, \(B_k\) is often available as an abstract operator rather than an explicit matrix.
In such a situation, computing \(\|B_k\|_1\), \(\|B_k\|_{\infty}\) or \(\|B_k\|_F\), say, is impractical.

We begin the convergence analysis by showing that there still exists a \(\Delta_{\mathrm{succ}}\) as in \citep[Theorem~\(3.4\)]{aravkin-baraldi-orban-2022}, despite our more general \Cref{asm:model-ubd}.

\begin{theorem}
  Let \Cref{asm:models}, \Cref{asm:model-diff} and \Cref{asm:model-ubd} be satisfied and
  \[
    \Delta_{\mathrm{succ}} := \frac{\kappa_{\mathrm{mdc}}(1 - \eta_2)}{2 \kappa_{\mathrm{ubd}} \alpha \beta^2} > 0.
  \]
  If~\eqref{eq:nlp} satisfies the constraint qualification at \(x_k\),~\eqref{eq:def-model-nu} satisfies the constraint qualification at \(0\), \(x_k\) is not first-order stationary for~\eqref{eq:nlp}, and \(\Delta_k \leq \Delta_{\mathrm{succ}}\), then iteration \(k\) is very successful and \(\Delta_{k+1} \ge \Delta_k\).
\end{theorem}

\begin{proof}
  By~\eqref{eq:xi1-suff-decrease} and~\eqref{eq:nu-ineq},
  \begin{equation}%
    \label{eq:suff-decrease}
    \xi_{\mathrm{cp}}(\Delta_k; x_k, \nu_k) \geq \tfrac{1}{2} \nu_k^{-1} \|s_{k,1}\|^2 \geq \tfrac{1}{2} (\alpha^{-1}\Delta_k^{-1} + \|B_k\|(1 + \alpha^{-1}\Delta_k^{-1})) \|s_{k,1}\|^2 \geq \tfrac{1}{2} (\alpha^{-1}\Delta_k^{-1}(1 + \|B_k\|)) \|s_{k,1}\|^2.
  \end{equation}
  If \(\xi_{\mathrm{cp}}(\Delta_k; x_k, \nu_k) = 0\), then \(s_{k,1} = 0\), and \(x_k\) is first-order stationary with \Cref{prop:iprox-stationarity}.
  If \(x_k\) is not first-order stationary, \(s_{k,1} \neq 0\) according to \Cref{prop:xicp-stat-0}.
  In this case, \Cref{asm:model-diff}, \Cref{asm:model-ubd}, and~\eqref{eq:suff-decrease} lead to
  \begin{equation*}
    \begin{aligned}
      |\rho_k - 1| & = \left| \frac{(f + h)(x_k + s_k) - m(s_k; x_k, B_k)}{m(0; x_k, B_k) - m(s_k; x_k, B_k)} \right|                                                             \\
                   & \le \frac{\kappa_{\mathrm{ubd}}(1 +\|B_k\|)\|s_k\|_2^2}{\kappa_{\mathrm{mdc}} \xi_{\mathrm{cp}}(\Delta_k; x_k, \nu_k)}                                       \\
                   & \le \frac{\kappa_{\mathrm{ubd}}(1 +\|B_k\|) \beta^2\|s_{k, 1}\|_2^2}{\tfrac{1}{2} \kappa_{\mathrm{mdc}} \alpha^{-1}\Delta_k^{-1}(1 + \|B_k\|) \|s_{k,1}\|^2} \\
                   & = \frac{2\kappa_{\mathrm{ubd}} \beta^2 \alpha \Delta_k}{\kappa_{\mathrm{mdc}}}.
    \end{aligned}
  \end{equation*}
  Thus, \(\Delta_k \le \Delta_{\mathrm{succ}}\) implies \(\rho_k \ge \eta_2\) and iteration \(k\) is very successful.
\end{proof}

We set \(\Delta_{\min} := \min(\Delta_0, \, \gamma_1 \Delta_{\mathrm{succ}})\), and we observe that \(\Delta_k \ge \Delta_{\min}\) for all \(k \in \N\).
Motivated by \Cref{prop:smooth-stat}, we use \(\nu_k^{-1/2}\xi_{\mathrm{cp}}(\Delta_k; x_k, \nu_k)^{1/2}\) as our criticality measure.
Let \(0 < \epsilon < 1\), \(k_\epsilon\) be the first iteration such that \(\nu_k^{-1/2} \xi_{\mathrm{cp}}(\Delta_k; x_k; \nu_k)^{1/2} \leq \epsilon\), and
\begin{align*}
  S(\epsilon) & := \{ k = 0, \ldots, k_\epsilon - 1 \mid \rho_k \geq \eta_1 \}, \\
  U(\epsilon) & := \{ k = 0, \ldots, k_\epsilon - 1 \mid \rho_k < \eta_1 \},
\end{align*}
be the set of successful, and unsuccessful iterations until the criticality measure drops below \(\epsilon\), respectively.

At iteration \(k\) of \Cref{alg:tr-nonsmooth}, let \(\sigma_k\) be the number of successful iterations encountered so far:
\begin{equation}%
  \label{eq:def-sigmak}
  \sigma_k = | \{ j = 0, \ldots, k \mid \rho_j \geq \eta_1 \} |, \quad k \in \N.
\end{equation}
We introduce an assumption allowing \(\{B_k\}\) to be unbounded, as long as it is controlled by \(\sigma_k\).

\begin{assumption}%
  \label{asm:hess-approx-growth}
  There are constants \(\mu > 0\) and \(0 \leq p < 1\) such that \(\max_{0 \leq j \leq k} \|B_j\| \leq \mu (1 + \sigma_k^p)\) for all \(k \in \N\).
\end{assumption}

Clearly, \Cref{asm:hess-approx-growth} allows approximations that grow unbounded, though they must not grow too fast.
It reduces to the bounded case when \(p = 0\).
Following the discussion in the introduction, it is possible that quasi-Newton approximations satisfy \Cref{asm:hess-approx-growth}, though that remains to be established.
The bound \(\|B_{j+1}\| \leq \|B_j\| + \kappa_B\) provided by \citet{conn-gould-toint-2000} and \citet{powell-2010} for the BFGS, SR1 and PSB updates, where \(\kappa_B > 0\) is a constant, suggests that in the worst case, certain quasi-Newton approximations could satisfy \Cref{asm:hess-approx-growth} with \(p = 1\).
Unfortunately, in that case, the analysis below would not apply to them.
However, once again, to the best of our knowledge, no bound on quasi-Newton approximations is known at this time.

Note also that we do not consider \(p > 1\) as~\eqref{eq:asm-sum-invmaxBj-inf} might no longer hold, and that would endanger convergence altogether.
We also do not consider here a variant of \Cref{asm:hess-approx-growth} in which \(\sigma_k\) is replaced with \(k\) because model Hessians are typically not updated on unsuccessful iterations---we are not aware of any algorithm that does, though it could of course be done.

We may now establish a variant of \citep[Lemma~\(3.6\)]{aravkin-baraldi-orban-2022} based on \Cref{asm:hess-approx-growth}.
The proof uses a technique similar to that of \citep[Lemma~\(3.6\)]{aravkin-baraldi-orban-2022}, itself inspired from the proofs of \citep{cartis-gould-toint-2022}, except for the management of \Cref{asm:hess-approx-growth}.
When \(p = 0\), \citet{aravkin-baraldi-orban-2022} show that \(|S(\epsilon)| = O(\epsilon^{-2})\).
In the following result, we restrict our attention to the case \(p > 0\).

\begin{lemma}%
  \label{lem:num-successful-unbounded}
  Let \Cref{asm:model-diff} and \Cref{asm:hess-approx-growth} be satisfied with \(p > 0\).
  Assume that \Cref{alg:tr-nonsmooth} generates infinitely many successful iterations when~\Cref{step:stop-crit} is ignored, that the step size \(\nu_k := \alpha \Delta_k / (1 + \|B_k\|(1 + \alpha \Delta_k))\) is selected at each iteration, and that there exists \((f + h)_{\mathrm{low}} \in \R\) such that \((f + h)(x_k) \geq (f + h)_{\mathrm{low}}\) for all \(k \in \N\).
  Let \(\epsilon \in (0, \, 1)\) be small enough that \(\mu + 1 \le \mu |S(\epsilon)|^p\).
  Then,
  \begin{equation}%
    \label{eq:qn-complexity-p}
    |S(\epsilon)| \leq
    \left( 2 \mu (1 + \alpha^{-1} \Delta_{\min}^{-1}) \, \frac{(f + h)(x_0) - (f + h)_{\mathrm{low}}}{\eta_1 \kappa_{\mathrm{mdc}} \epsilon^2} \right)^{1 / (1 - p)} = O\left(\epsilon^{-2/(1-p)}\right).
  \end{equation}
\end{lemma}

\begin{proof}
  Let \(k \in S(\epsilon)\).
  We proceed as in \citep[Lemma~\(3.6\)]{aravkin-baraldi-orban-2022} with the minor corrections made in \citep{aravkin-baraldi-leconte-orban-2023}.
  We have
  \begin{align*}
    (f + h)(x_k) - (f + h)(x_k + s_k) & \geq \eta_1 \kappa_{\mathrm{mdc}} \xi_{\mathrm{cp}}(\Delta_k; x_k, \nu_k)
    \\
                                      & \geq \eta_1 \kappa_{\mathrm{mdc}} \nu_k \epsilon^2
    \\
                                      & = \eta_1 \kappa_{\mathrm{mdc}} \frac{1}{\alpha^{-1} \Delta_k^{-1} + \|B_k\|(1 + \alpha^{-1} \Delta_k^{-1})} \epsilon^2
    \\
                                      & \geq \eta_1 \kappa_{\mathrm{mdc}} \frac{1}{\alpha^{-1} \Delta_{\min}^{-1} + \|B_k\|(1 + \alpha^{-1} \Delta_{\min}^{-1})} \epsilon^2.
  \end{align*}
  We add together the above inequalities over all \(k \in S(\epsilon)\) and use the assumption that \(f + h\) is bounded below to obtain
  \begin{align*}
    (f + h)(x_0) - (f + h)_{\mathrm{low}} & \geq \eta_1 \kappa_{\mathrm{mdc}} \epsilon^2 \sum_{k \in S(\epsilon)} \frac{1}{\alpha^{-1} \Delta_{\min}^{-1} + \|B_k\|(1 + \alpha^{-1} \Delta_{\min}^{-1})}
    \\
                                          & \geq \eta_1 \kappa_{\mathrm{mdc}} \epsilon^2 |S(\epsilon)| \min_{k \in S(\epsilon)} \frac{1}{\alpha^{-1} \Delta_{\min}^{-1} + \|B_k\|(1 + \alpha^{-1} \Delta_{\min}^{-1})}
    \\
                                          & = \eta_1 \kappa_{\mathrm{mdc}} \epsilon^2 |S(\epsilon)| \frac{1}{\max_{k \in S(\epsilon)} (\alpha^{-1} \Delta_{\min}^{-1} + \|B_k\|(1 + \alpha^{-1} \Delta_{\min}^{-1}))}
    \\
                                          & = \eta_1 \kappa_{\mathrm{mdc}} \epsilon^2 |S(\epsilon)| \frac{1}{\alpha^{-1} \Delta_{\min}^{-1} + (\max_{k \in S(\epsilon)} \|B_k\|) (1 + \alpha^{-1} \Delta_{\min}^{-1})}
    \\
                                          & \geq \eta_1 \kappa_{\mathrm{mdc}} \epsilon^2 |S(\epsilon)| \frac{1}{\alpha^{-1} \Delta_{\min}^{-1} + \mu( 1 + |S(\epsilon)|^p) (1 + \alpha^{-1} \Delta_{\min}^{-1})}
    \\
                                          & \geq \eta_1 \kappa_{\mathrm{mdc}} \epsilon^2 |S(\epsilon)| \frac{1}{(\mu + 1 + \mu |S(\epsilon)|^p) (1 + \alpha^{-1} \Delta_{\min}^{-1})},
  \end{align*}
  where we appealed to \Cref{asm:hess-approx-growth} in the penultimate step.

  Because, \(\mu + 1 \le \mu |S(\epsilon)|^p\),
  \[
    (f + h)(x_0) - (f + h)_{\mathrm{low}} \geq \eta_1 \kappa_{\mathrm{mdc}} \epsilon^2 |S(\epsilon)| \frac{1}{2\mu|S(\epsilon)|^p (1 + \alpha^{-1} \Delta_{\min}^{-1})} = \eta_1 \kappa_{\mathrm{mdc}} \epsilon^2 |S(\epsilon)|^{1-p} \frac{1}{2\mu(1 + \alpha^{-1} \Delta_{\min}^{-1})},
  \]
  which establishes~\eqref{eq:qn-complexity-p}.
\end{proof}

If there are infinitely many successful iterations, the inequality \(\mu + 1 > \mu |S(\epsilon)|^p\) can only hold for all sufficiently small \(\epsilon > 0\) if \(p = 0\).

The complexity bound \(|S(\epsilon)| = O(\epsilon^{-2/(1-p)})\) also holds for \(p = 0\), as it reduces to that of \citet{aravkin-baraldi-orban-2022}.
We obtain a complexity bound of \(O(\epsilon^{-5/2})\) for \(p = \tfrac{1}{5}\) and \(O(\epsilon^{-3})\) for \(p = \tfrac{1}{3}\).
In other words, the faster the growth of \(\|B_k\|\), the worse the deterioration of the complexity bound.

A bound on the number of unsuccessful iterations is obtained using the technique of \citet{cartis-gould-toint-2022}.

\begin{proposition}[{\protect \citealp[Lemma~\(3.7\)]{aravkin-baraldi-orban-2022}}]%
  \label{prop:num-unsuccessful-unbounded}
  Under the assumptions of \Cref{lem:num-successful-unbounded},
  \begin{equation}
    \label{eq:number-unsuccessful-ubd}
    |U(\epsilon)| \leq \log_{\gamma_2} (\Delta_{\min} / \Delta_0) + |S(\epsilon)| \, |\log_{\gamma_2}(\gamma_4)|.
  \end{equation}
\end{proposition}

\begin{proof}
  The proof is a minor modification of that of \citep[Lemma~\(3.7\)]{aravkin-baraldi-orban-2022}.
  We provide it for completeness.
  The update rule of \(\Delta_k\) in \Cref{step:deltak-update} indicates that
  \begin{equation*}
    \Delta_{\min} \le \Delta_{k_{\epsilon}-1} \le \min(\Delta_0 \gamma_2^{|U(\epsilon)|} \gamma_4^{|S(\epsilon)|}, \, \Delta_{\max}) \le \Delta_0 \gamma_2^{|U(\epsilon)|} \gamma_4^{|S(\epsilon)|}.
  \end{equation*}
  As \(0 < \gamma_2 < 1\), we take the logarithm of the above inequalities to obtain
  \begin{equation*}
    |U(\epsilon)| \log(\gamma_2) + |S(\epsilon)|\log(\gamma_4) \ge \log(\Delta_{\min} / \Delta_0),
  \end{equation*}
  which leads to~\eqref{eq:number-unsuccessful-ubd}.
\end{proof}

We caution the reader that as \(p \uparrow 1\), \Cref{lem:num-successful-unbounded} does not provide any useful bound.
Thus, our analysis really only applies to fixed \(p < 1\) and no limit should be taken in~\eqref{eq:qn-complexity-p}.
However, as \citet{powell-2010} surmises in his concluding remarks, the number of iterations should remain finite, ``monstrous'' though it may be, when \(p = 1\).
Specifically, \citeauthor{powell-2010} conjectures a pessimistic bound of the form \(O(\exp(\exp(1/\epsilon)))\).

The following result follows from \Cref{lem:num-successful-unbounded} and \Cref{prop:num-unsuccessful-unbounded}.

\begin{corollary}%
  \label{cor:liminf-xicp}
  Under the assumptions of \Cref{lem:num-successful-unbounded}, \(\liminf \nu_k^{-1/2} \xi_{\mathrm{cp}}(\Delta_k; x_k, \nu_k)^{1/2} = 0\).
\end{corollary}

\begin{proof}
  The first part of this result is obtained similarly as in \citep[Theorem~\(3.8\)]{aravkin-baraldi-orban-2022}.
  \Cref{lem:num-successful-unbounded} and \Cref{prop:num-unsuccessful-unbounded} indicate that
  \begin{equation*}
    |S(\epsilon)| + |U(\epsilon)| = O \left(\epsilon^{-2/(1-p)} \right),
  \end{equation*}
  thus \(\liminf \nu_k^{-1/2} \xi_{\mathrm{cp}}(\Delta_k; x_k, \nu_k)^{1/2} = 0\).
\end{proof}

To conclude this section, we examine conditions under which limit points of \(\{x_k\}\) are first-order stationary for~\eqref{eq:nlp}.
We first establish results about the first-order stationarity conditions of~\eqref{eq:ipg-iter-sk1}.

\begin{lemma}%
  \label{lem:sk1-to-0}
  Under the assumptions of \Cref{lem:num-successful-unbounded}, there exists an infinite index set \(N \subseteq \N\) such that
  \begin{enumerate}
    \item\label{itm:xi/nu-to-0}%
      \(\{\nu_k^{-1/2} \xi_{\mathrm{cp}}(\Delta_k; x_k, \nu_k)^{1/2}\}_N \to 0\),
    \item\label{itm:s1-to-0}%
      \(\{\nu_k^{-1} s_{k,1}\}_{k \in N} \to 0\), and therefore, \(\{s_{k,1}\}_N \to 0\), and
    \item\label{itm:s-to-0}%
      \(\{s_k\}_N \to 0\).
  \end{enumerate}
\end{lemma}

\begin{proof}
  \Cref{cor:liminf-xicp} ensures the existence of an infinite index set \(N \subseteq \N\) such that claim~\ref{itm:xi/nu-to-0} holds.
  By~\eqref{eq:xi1-suff-decrease},
  \(
  \nu_k^{-1/2} \xi_{\mathrm{cp}}(\Delta_k; x_k, \nu_k)^{1/2} \ge \nu_k^{-1} \|s_{k,1}\| / \sqrt{2}
  \).
  Thus, \(\{\nu_k^{-1} s_{k,1}\}_N \to 0\).
  As \(\liminf \nu_k \geq 0\) and \(\sup \nu_k < \infty\) always hold, claim~\ref{itm:s1-to-0} must hold.
  Because \(\|s_k\| \leq \beta \|s_{k,1}\|\) for all \(k\) by \Cref{step:sk-computation} of \Cref{alg:tr-nonsmooth}, claim~\ref{itm:s-to-0} holds.
\end{proof}

Because \(\Delta_k \ge \Delta_{\min} > 0\) for all \(k\), and \(\{s_{k,1}\}_N \to 0\) by \Cref{lem:sk1-to-0}, there exists \(k_0 \in N\) such that for all \(k \in N\) with \(k \ge k_0\), \(s_{k,1}\) is not on the boundary of \(\Delta_k \B\).
By~\eqref{eq:ipg-iter-sk1}, we have for all \(k \in N\) with \(k \geq k_0\),
\begin{equation}%
  \label{eq:sk1-no-tr}
  s_{k,1} \in \argmin{s} \ m_{\mathrm{cp}}(s; x_k, \nu_k) + \chi(x_k + s \mid [\ell, \, u]).
\end{equation}

In the following, we define, for all \(x\) and \(s \in \R^n\),
\begin{equation}%
  \label{eq:def-psi-hat}
  \widehat{\psi}(s; x) \coloneq \psi(s; x) + \chi(x + s \mid [\ell, u]).
\end{equation}

\begin{lemma}%
  \label{lem:s1-stationarity}
  Let \(N\) be the infinite index set of \Cref{lem:sk1-to-0}.
  Then, there exists \(k_0 \in N\) such that for all \(k \in N\) with \(k \ge k_0\),
  \begin{equation}
    \label{eq:s1-stationarity}
    -\nu_k^{-1} s_{k,1} \in \nabla f(x_k) + \partial \widehat{\psi}(s_{k,1}; x_k).
  \end{equation}
\end{lemma}

\begin{proof}
  The claim follows directly from the first-order stationarity conditions of~\eqref{eq:sk1-no-tr}.
\end{proof}

In view of \Cref{lem:sk1-to-0,lem:s1-stationarity}, for all \(\epsilon > 0\), there exists \(k_\epsilon \in \N\) such that for all \(k \geq k_\epsilon\) with \(k \in N\), there is \(u_k \in \partial \widehat{\psi}(s_{k,1}; x_k)\) satisfying
\[
  \|\nabla f(x_k) + u_k\| \leq \epsilon.
\]
The above suggests that limit points of \(\{(x_k, u_k)\}_{k \in N}\) may be expected to be stationary for~\eqref{eq:nlp} under certain conditions.
We now make this last statement more precise.

When \(\liminf \nu_k > 0\), which happens when \(\{B_k\}\) remains bounded, and when models \(\psi\) are lsc in the joint variables \((s, x)\), \citep[Proposition~\(3.10\)]{aravkin-baraldi-orban-2022} established that \(\xi_{\mathrm{cp}}(\Delta_{\min}; \cdot, \cdot)\) is lsc and that if \((\bar{x}, \bar{\nu})\) is a limit point of \(\{(x_k, \nu_k)\}\), then \(\bar{x}\) is first-order stationary for~\eqref{eq:nlp}.
However, that result does not take explicit bound constraints into account.

We now provide an alternative analysis before examining the case where \(\{\nu_k\} \to 0\).

If \(\liminf_{k \in N} \nu_k > 0\), there exists an infinite index \(N_1 \subseteq N\) such that \(\{\nu_k\}_{k \in N_1} \to \bar{\nu} > 0\).
The following results hinge around epigraphical convergence \citep[Chapter~\(7\)]{rockafellar-wets-1998} and consist in determining the epigraphical limit of the sequence of models.

Consider the situation where \(\{x_k\}_{k \in N_1}\) has a limit point, or, without loss of generality, that \(\{x_k\}_{k \in N_1} \to \bar{x}\).
It does not follow that \(\{\chi(x_k + \cdot \mid [\ell, \, u])\}_{k \in N_1} \to \chi(\bar{x} + \cdot \mid [\ell, \, u])\) pointwise or continuously.
Indeed, if \(\bar x + s\) is on the boundary of \([\ell, \, u]\) and \(x_k + s_k\) lies outside of \([\ell, \, u]\) for all \(k \in N_1\) with \(\{x_k + s_k\}_{k \in N_1} \to \bar x + s\), \(\{\chi(x_k + s_k \mid [\ell, \, u])\}_{k \in N_1} \to +\infty\) while \(\chi(\bar{x} + s \mid [\ell, \, u]) = 0\).
However, convergence occurs epigrahically.

\begin{lemma}%
  \label{lem:elim-chik}
  Let \(N\) be the infinite index set of \Cref{lem:sk1-to-0}.
  Let \(\{x_k\}_{k \in N_1} \to \bar{x} \in [\ell, \, u]\), where \(N_1 \subseteq N\) is defined as above.
  Then
  \[
    \elim_{k \in N_1} \chi(x_k + \cdot \mid [\ell, \, u]) = \chi(\bar x + \cdot \mid [\ell, \, u]).
  \]
\end{lemma}

\begin{proof}
  The result follows from \citep[Proposition~\(7.4f\)]{rockafellar-wets-1998} applied to \(C_k := x_k + [\ell, u]\) and \(C := \bar x + [\ell, u]\), since \(\{C_k\}_{k \in N_1} \to C\).
\end{proof}

\begin{theorem}%
  \label{thm:cont-conv}
  Let \(N\) be the infinite index set of \Cref{lem:sk1-to-0}.
  Let \(\{x_k\}_{k \in N_1} \to \bar{x} \in [\ell, \, u]\), where \(N_1 \subseteq N\) is defined as above.
  Assume that there is \(\widebar{\psi} : \R^n \to \widebar{\R}\) such that \(\{\psi(\cdot; x_k)\}_{k \in N_1} \to \widebar{\psi}\) continuously, and that satisfies \Cref{asm:models} as a model about \(\bar{x}\), i.e., \(\widebar{\psi}(0) = h(\bar{x})\) and \(\partial \widebar{\psi}(0) \subseteq \partial h(\bar{x})\).
  Assume further that the constraint qualification~\eqref{eq:cq} is satisfied at \(s = 0\) for
  \begin{equation*}%
    \minimize{s} \ \widebar{m}_{\mathrm{cp}}(s; \bar{x}, \bar{\nu}) + \chi(\bar{x} + s \mid [\ell, \, u]),
    \qquad
    \widebar m_{\mathrm{cp}}(s; \bar{x}, \bar{\nu}) \coloneq \varphi_{\mathrm{cp}}(s; \bar{x}) + \tfrac{1}{2} \bar{\nu}^{-1} \|s\|^2 + \widebar\psi(s),
  \end{equation*}
  and that it is satisfied at \(\bar x\) for~\eqref{eq:nlp}.
  If \(-\infty < \inf_s \widebar{m}_{\mathrm{cp}}(s; \bar{x}, \bar{\nu}) + \chi(\bar{x} + s \mid [\ell, \, u]) < \infty\), \(\bar{x}\) is stationary for~\eqref{eq:nlp}.
\end{theorem}

\begin{proof}
  Continuity of \(\nabla f\) and \citep[Theorem~\(7.11a\) and \(b\)]{rockafellar-wets-1998} ensure that
  \begin{equation}%
    \label{eq:elim-continuous-terms}
    \elim_{k \in N_1} \varphi_{\mathrm{cp}}(\cdot; x_k) + \tfrac{1}{2} \nu_k^{-1} \|\cdot\|^2 = \varphi_{\mathrm{cp}}(\cdot; \bar{x}) + \tfrac{1}{2} \bar{\nu}^{-1} \|\cdot\|^2,
  \end{equation}
  and the convergence is continuous.
  By \Cref{lem:elim-chik} and \citep[Theorem~\(7.46b\)]{rockafellar-wets-1998},
  \[
    \elim_{k \in N_1} \varphi_{\mathrm{cp}}(\cdot; x_k) + \tfrac{1}{2} \nu_k^{-1} \|\cdot\|^2 + \chi(x_k + \cdot \mid [\ell, \, u]) = \varphi_{\mathrm{cp}}(\cdot; \bar{x}) + \tfrac{1}{2} \bar{\nu}^{-1} \|\cdot\|^2 + \chi(\bar{x} + \cdot \mid [\ell, \, u]).
  \]
  Again, \citep[Theorem~\(7.46b\)]{rockafellar-wets-1998} and the continuous convergence of \(\{\psi(\cdot; x_k)\}_{k \in N_1}\) yield
  \[
    \elim_{k \in N_1} m_{\mathrm{cp}}(\cdot; x_k, \nu_k) + \chi(x_k + \cdot \mid [\ell, \, u]) = \widebar m_{\mathrm{cp}}(\cdot; \bar{x}, \bar{\nu}) + \chi(\bar{x} + \cdot \mid [\ell, \, u]).
  \]
  Because \(s_{k,1} \in \argmin{s} m_{\mathrm{cp}}(s; x_k, \nu_k) + \chi(x_k + s \mid [\ell, \, u])\) for all \(k \in N_1\) and \(\{s_{k,1}\}_{k \in N_1} \to 0\) by \Cref{lem:sk1-to-0} and~\eqref{eq:sk1-no-tr}, we obtain from \citep[Theorem~\(7.31b\)]{rockafellar-wets-1998} that
  \[
    0 \in \argmin{s} \widebar m_{\mathrm{cp}}(s; \bar{x}, \bar{\nu}) + \chi(\bar{x} + s \mid [\ell, \, u]),
  \]
  which implies that \(\bar{x}\) is stationary for~\eqref{eq:nlp}.
\end{proof}

For the limiting model \(\widebar \psi\) of \Cref{thm:cont-conv} to satisfy \Cref{asm:models}, we must have \(\widebar \psi(0) = h(\bar{x})\) and \(\partial \widebar \psi(0; \bar{x}) \subseteq \partial h(\bar{x})\).
We now review two important examples in practice.

A widely used model is simply \(\psi(s; x_k) = h(x_k + s)\) for all \(k \in \N\).
Clearly, when \(h\) is continuous, the limiting model satisfies \Cref{asm:models}.
A common situation occurs when \(h(x) = g(c(x))\), where \(c: \R^n \to \R^m\) is \(\mathcal{C}^1\) and \(g: \R^m \to \R\) is continuous.
In penalty scenarios, \(g\) is a norm.
It is then natural to choose \(\psi(s; x_k) := g(c(x_k) + \nabla c(x_k) s)\) for all \(k\).
Again, the limiting model satisfies \Cref{asm:models}.

In the absence of bound constraints, we may weaken the assumption on continuous convergence of \(\{\psi(\cdot; x_k)\}\) in \Cref{thm:cont-conv}.
We require another technical lemma.
Recall first that for any sequence \(\{f_k\}\) where \(f_k: \R^n \to \widebar{\R}\), \citep[Proposition~\(7.2\)]{rockafellar-wets-1998} states that
\begin{equation}%
  \label{eq:elimsup}
    (\elimsup f_k)(x) = \min \{\alpha \in \widebar{\R} \mid \exists \{x_k\} \to x, \ \limsup f_k(x_k) = \alpha \},
    \quad (x \in \R^n).
\end{equation}
Note that the \(\min\) in~\eqref{eq:elimsup} is attained.

\begin{lemma}%
  \label{lem:sum-elimsup}
  For \(k \in \N\), let \(\phi_k\), \(\phi : \R^n \to \R\),\footnote{Note that \(\phi_k\) and \(\phi\) take finite values.} and \(\psi_k\), \(\widebar{\psi} : \R^n \to \widebar{\R}\).
  Assume that \(\{\phi_k\} \to \phi\) continuously, and that \(\elimsup \psi_k = \widebar{\psi}\).
  Then, \(\elimsup \phi_k + \psi_k = \phi + \widebar{\psi}\).
\end{lemma}

\begin{proof}
  Let \(x \in \R^n\).
  For any sequence \(\{x_k\}\),
  \[
    \liminf \phi_k(x_k) + \limsup \psi_k(x_k) \leq \limsup (\phi_k(x_k) + \psi_k(x_k)) \leq \limsup \phi_k(x_k) + \limsup \psi_k(x_k).
  \]
  In particular, for any sequence \(\{x_k\} \to x\), continuous convergence of \(\{\phi_k\}\) to \(\phi\) implies that \(\liminf \phi_k(x_k) = \limsup \phi_k(x_k) = \lim \phi_k(x_k) = \phi(x)\).
  Thus,
  \[
    \phi(x) + \limsup \psi_k(x_k) = \limsup (\phi_k(x_k) + \psi_k(x_k)).
  \]
  By~\eqref{eq:elimsup}, there is a sequence \(\{x_k^{\phi+\psi}\} \to x\) such that \(\limsup (\phi_k(x_k^{\phi+\psi}) + \psi_k(x_k^{\phi+\psi})) = (\elimsup \phi_k + \psi_k)(x)\).
  For that sequence,~\eqref{eq:elimsup} also shows that \(\limsup \psi_k(x_k^{\phi+\psi}) \geq \widebar{\psi}(x)\).
  Hence,
  \[
    \phi(x) + \widebar{\psi}(x) \leq \elimsup (\phi_k + \psi_k)(x).
  \]
  Similarly, there is a sequence \(\{x_k^\psi\} \to x\) such that \(\limsup \psi_k(x_k^\psi) = \widebar{\psi}(x)\).
  For that sequence, we now have \(\limsup (\phi_k(x_k^\psi) + \psi_k(x_k^\psi)) \geq (\elimsup \phi_k + \psi_k)(x)\), and hence,
  \[
    \phi(x) + \widebar{\psi}(x) \geq \elimsup (\phi_k + \psi_k)(x).
  \]
  Since \(x\) was arbitrary, we obtain the claim.
\end{proof}

In the following result, continuous convergence of \(\{\psi(\cdot; x_k)\}_{k \in N_1}\) is replaced with existence of the epigraphical \(\limsup\) and continuous convergence with respect to \(\{s_{k,1}\}_{k \in N_1} \to 0\).
The relevance of the \(\elimsup\) in this context stems from \citep[Proposition~\(7.30\)]{rockafellar-wets-1998}.

\begin{theorem}%
  \label{thm:econv-sup}
  Assume~\eqref{eq:nlp} has no bound constraints.
  Let \(N\) be the infinite index set of \Cref{lem:sk1-to-0}, and \(\{x_k\}_{k \in N_1} \to \bar{x}\), where \(N_1 \subseteq N\) is defined as above.
  Assume
  \[
    \widebar{\psi} \coloneq \elimsup_{k \in N_1} \psi(\cdot; x_k)
  \]
  is not identically \(+\infty\) and satisfies \Cref{asm:models} as a model about \(\bar{x}\), i.e., \(\widebar{\psi}(0) = h(\bar{x})\) and \(\partial \widebar{\psi}(0) \subseteq \partial h(\bar{x})\).
  If \(\{\psi(s_{k,1}; x_k)\}_{k \in N_1} \to \widebar{\psi}(0)\), then \(\bar{x}\) is stationary for~\eqref{eq:nlp}.
\end{theorem}

\begin{proof}
  As in the proof of \Cref{thm:cont-conv},~\eqref{eq:elim-continuous-terms} holds.
  \Cref{lem:sum-elimsup} yields
  \[
    \elimsup_{k \in N_1} m_{\mathrm{cp}}(\cdot; x_k, \nu_k) = \widebar m_{\mathrm{cp}}(\cdot; \bar{x}, \bar{\nu}),
    \qquad
    \widebar m_{\mathrm{cp}}(s; \bar{x}, \bar{\nu}) \coloneq \nabla f(\bar{x})^T s + \tfrac{1}{2} \bar{\nu}^{-1} \|s\|^2 + \widebar\psi(s).
  \]
  If \(\{\psi(s_{k,1}; x_k, \nu_k)\}_{k \in N_1} \to \widebar\psi(0)\), then \(\{m_{\mathrm{cp}}(s_{k,1}; x_k, \nu_k)\}_{k \in N_1} \to \widebar m_{\mathrm{cp}}(0; \bar{x}, \bar{\nu})\).
  By \citep[Proposition~\(7.30\)]{rockafellar-wets-1998}, we obtain that \(0 \in \argmin{s} \widebar m_{\mathrm{cp}}(s; \bar{x}, \bar{\nu})\), which implies that \(\bar{x}\) is stationary for~\eqref{eq:nlp}.
\end{proof}

Finally, we may trade the continuous convergence of \(\{\psi(\cdot; x_k)\}_{k \in N_1}\) with respect to  \(\{s_{k,1}\}_{k \in N_1}\) for the existence of the epigraphical limit of the models \(\{\psi(\cdot; x_k)\}_{k \in N_1}\) and their pointwise convergence to that limit.

\begin{theorem}%
  \label{thm:econv}
  Assume~\eqref{eq:nlp} has no bound constraints.
  Let \(N\) be the infinite index set of \Cref{lem:sk1-to-0}, and \(\{x_k\}_{k \in N_1} \to \bar{x}\), where \(N_1 \subseteq N\) is defined as above.
  Assume
  \[
    \widebar{\psi} \coloneq \elim_{k \in N_1} \psi(\cdot; x_k)
  \]
  exists and satisfies \Cref{asm:models} as a model about \(\bar{x}\), i.e., \(\widebar{\psi}(0) = h(\bar{x})\) and \(\partial \widebar{\psi}(0) \subseteq \partial h(\bar{x})\).
  Assume further that \(-\infty < \inf \widebar\psi < +\infty\), and that \(\lim_{k \in N_1} \psi(s; x_k) = \widebar{\psi}(s; \bar{x})\) for all \(s \in \R^n\).
  Then \(\bar{x}\) is stationary for~\eqref{eq:nlp}.
\end{theorem}

\begin{proof}
  As in the proof of \Cref{thm:cont-conv},~\eqref{eq:elim-continuous-terms} holds.
  Since \(\varphi_{\mathrm{cp}}(s, \bar{x}) > -\infty\) and \(\widebar{\psi}(s) > -\infty\) for all \(s \in \R^n\), our assumptions guarantee that \citep[Theorem~\(3.2\), conditions~(I) and (i)]{burke-hoheisel-2015} are satisfied for all \(s \in \R^n\).
  Therefore, \(\elim_{k \in N_1} m_{\mathrm{cp}}(\cdot; x_k, \nu_k) = \widebar m_{\mathrm{cp}}(\cdot; \bar{x}, \bar{\nu})\).
  Because \(-\infty < \inf \widebar m_{\mathrm{cp}}(\cdot; \bar{x}, \bar{\nu}) < \infty\), \(s_{k,1} \in \argmin{s} m_{\mathrm{cp}}(s; x_k, \nu_k)\) for all \(k \in N_1\), and recalling point~\ref{itm:s1-to-0} of \Cref{lem:sk1-to-0}, \citep[Theorem~\(7.31b\)]{rockafellar-wets-1998} implies that \(0 \in \argmin{s} \widebar m_{\mathrm{cp}}(s; \bar{x}, \bar{\nu})\), which proves that \(\bar{x}\) is stationary for~\eqref{eq:nlp}.
\end{proof}

When \(\liminf \nu_k\) may be zero, i.e., when \(\{B_k\}\) may not be bounded, we work directly with subdifferentials.

\begin{modelassumption}%
  \label{asm:subgradient-ensemble-limit-hat}
  There exists a model \(\psi(\cdot; \bar{x})\) that satisfies \Cref{asm:models} such that, for subsequences \(\{s_{k,1}\}_N \to 0\) and \(\{x_k\}_N \to \bar x \in [\ell, u]\) such that for all \(k \in N\), \(x_k + s_{k,1} \in [\ell, u]\),
  \begin{equation}
    \limsup_{k \in N} \partial \widehat{\psi}(s_{k,1}; x_k) \subseteq \partial \widehat{\psi}(0; \bar x),
  \end{equation}
  where \(\widehat{\psi}\) is defined in~\eqref{eq:def-psi-hat}.
\end{modelassumption}

\Cref{asm:subgradient-ensemble-limit-hat} holds, among others, in the following cases:
\begin{enumerate}
  \item When \(\psi(\cdot; x_k)\) and \(\psi(\cdot; \bar x)\) are proper, lsc, convex functions with \(\psi(\cdot; x_k) \overset{e}{\to} \psi(\cdot; \bar x)\), and \(0\) lies in the interior of \(\dom \psi(\cdot; \bar x) - \dom \chi(\bar{x} + \cdot \mid [\ell, u]) = \{s_1 - s_2 \mid \psi(s_1; \bar{x}) < +\infty \text{ and } \ell \leq \bar{x} + s_2 \leq u\}\).
  Indeed, in that case, \(\widehat{\psi}(\cdot; x_k)\) and \(\widehat{\psi}(0; \bar{x})\) are also proper, lsc and convex, and \citep[Theorem~\(7.47b\)]{rockafellar-wets-1998} guarantees that \(\widehat{\psi}(\cdot; x_k) \overset{e}{\to} \widehat{\psi}(\cdot; \bar{x})\), thus Attouch's theorem \citep[Theorem~\(12.35\)]{rockafellar-wets-1998} holds.
  Extensions to non-convex functions under more sophisticated assumptions are established by \citet{penot-1995,poliquin-1992} and references therein.
  \item When \(\psi(s; x) = h(x + s)\) and \(h(x_k + s_{k,1}) \to h(\bar x)\), using \citep[Proposition~\(8.7\)]{rockafellar-wets-1998} applied to \(\{x_k + s_{k,1}\}_N \to \bar x\).
\end{enumerate}

We may now establish the following result.

\begin{theorem}%
  \label{thm:x-stationary}
  Let the assumptions of \Cref{lem:num-successful-unbounded} be satisfied.
  Let \Cref{asm:subgradient-ensemble-limit-hat} hold for the infinite index \(N\) of \Cref{lem:sk1-to-0}, and assume that \(\{x_k\}_N \to \bar x\).
  Assume further that the constraint qualification~\eqref{eq:cq} is satisfied at \(s = 0\) for
  \begin{equation}%
    \label{eq:limiting-model}
    \minimize{s} \ m_{\mathrm{cp}}(s; \bar{x}, \bar{\nu}) + \chi(\bar{x} + s \mid [\ell, \, u]),
  \end{equation}
  for some \(\bar{\nu} > 0\), and that it is satisfied at \(\bar x\) for~\eqref{eq:nlp}.
  Then, \(\bar x\) is first-order stationary.
\end{theorem}

\begin{proof}
  By \Cref{lem:s1-stationarity}, there exists \(u_k \in \partial \widehat{\psi}(s_{k,1}; x_k)\) for all \(k \in N\) such that \(-\nu_k^{-1} s_{k,1} = \nabla f(x_k) + u_k\).
  By \Cref{lem:sk1-to-0}, the convergence of \(\{x_k\}_N\) and continuity of \(\nabla f\), \(\{u_k\}\) converges.
  Let \(\bar{u}\) be its limit.
  In the limit over \(k \in N\), we obtain \(\bar{u} = -\nabla f(\bar{x})\).
  \Cref{asm:subgradient-ensemble-limit-hat} implies that \(\bar{u} \in \partial \widehat{\psi}(0; \bar{x})\).
  Because the constraint qualification is satisfied at \(s = 0\) for~\eqref{eq:limiting-model}, \citep[Corollary~\(10.9\)]{rockafellar-wets-1998} and \Cref{asm:models} yield \(\bar{u} \in \partial \psi(0; \bar{x}) + N_{[\ell, \, u]}(\bar{x}) \subseteq \partial h(\bar{x}) + N_{[\ell, \, u]}(\bar{x})\).
  Thus, the first-order stationarity conditions of~\eqref{eq:nlp} under~\eqref{eq:cq} hold.
\end{proof}

In \Cref{thm:x-stationary}, the value of \(\bar{\nu}\) is unimportant as it plays no role in the subdifferential of the objective of~\eqref{eq:limiting-model} at \(s = 0\).

\section{Sharpness of the complexity bound}%
\label{sec:sharp}

In this section, we show that the bound of \Cref{lem:num-successful-unbounded} is attained using the techniques of \citet[Theorem~\(2.2.3\)]{cartis-gould-toint-2022}.
Even though those authors only use said techniques to construct examples under the assumption that model Hessians remain bounded, they can be used under \Cref{asm:hess-approx-growth} as well because the number of values to interpolate before a stopping condition is met is always finite.
We have not seen those techniques used in the present context elsewhere in the literature.

For \(0 < \epsilon \le 1/2\), we explicitly construct \(k_{\epsilon} = \lfloor \epsilon^{-2/(1-p)} \rfloor\) iterates of \Cref{alg:tr-nonsmooth} with \(n = 1\) and \(h = 0\), so that \(\nu_k^{-1/2}\xi_{\mathrm{cp}}(\Delta_k; x_k, \nu_k)^{1/2} > \epsilon\) for \(k=0, \ldots, k_{\epsilon} - 1\), and \(\nu_{k_{\epsilon}}^{-1/2}\xi_{\mathrm{cp}}(\Delta_{k_{\epsilon}}; x_{k_{\epsilon}}, \nu_{k_{\epsilon}})^{1/2} = \epsilon\).
Then, we invoke \citep[Theorem~\(A.9.2\)]{cartis-gould-toint-2022} to establish that there exists \(f: \R \to \R\) in~\eqref{eq:nlp} that interpolates our iterates and satisfies our assumptions.
The following result is a special case of \citep[Theorem~\(A.9.2\)]{cartis-gould-toint-2022}.

\begin{proposition}[Hermite interpolation with function and gradient evaluations]%
  \label{prop:hermite-interpol}
  Let \(k_{\epsilon}\) be a positive integer, \(\{f_k\}\), \(\{g_k\}\) and \(\{x_k\}\) be sequences of numbers given for \(k \in \{0, \ldots, k_{\epsilon}\}\).
  Assume that for \(k \in \{0, \ldots, k_{\epsilon}\}\), \(s_k = x_{k+1} - x_k > 0\), and that for all \(k \in \{0, \ldots, k_{\epsilon}-1\}\),
  \begin{subequations}%
    \label{eq:fk-gk-ineq}
    \begin{align}
       & |f_{k+1} - (f_k + g_k s_k)| \le \kappa_f s_k^2, \\
       & |g_{k+1} - g_k| \le \kappa_f s_k,
    \end{align}
  \end{subequations}
  for some constant \(\kappa_f \ge 0\).
  Then, there exists \(f: \R \to \R\) continuously differentiable such that
  \begin{equation*}
    f(x_k) = f_k \quad \text{and} \quad f'(x_k) = g_k.
  \end{equation*}
  In addition, if
  \begin{equation*}
    |f_k| \le \kappa_f, \quad |g_k| \le \kappa_f \quad \text{and} \quad s_k \le \kappa_f,
  \end{equation*}
  then \(|f|\) and \(|f'|\) are bounded by a constant depending only on \(\kappa_f\).
\end{proposition}

\begin{proof}
  The result is a special case of \citep[Theorem~\(A.9.2\)]{cartis-gould-toint-2022} with \(p = 1\).
\end{proof}

In the following, we use
\begin{subequations}%
  \label{eq:def-eps-p-keps}
  \begin{align}
     & 0 < \epsilon \le 1/2,                               \\
     & 0 \le p < 1,                                        \\
     & k_{\epsilon} = \lfloor \epsilon^{-2/(1-p)} \rfloor, \\
     & \alpha > 0,                                         \\
     & \beta \ge 2 \alpha^{-1} + 1,
  \end{align}
\end{subequations}
and for all \(k \in \{0, \ldots, k_{\epsilon}\}\), we define the sequences
\begin{subequations}%
  \label{eq:def-wk-gk}
  \begin{align}
    w_k & := (k_{\epsilon} - k) / k_{\epsilon}, \\
    g_k & := - \epsilon (1 + w_k).
    \label{eq:def-gk}
  \end{align}
\end{subequations}
In addition, using the initial values
\begin{subequations}%
  \label{eq:def-val0}
  \begin{align}
    \Delta_0 & := 1,                             \\
    B_0      & := 1,
    \label{eq:val-B0}%
    \\
    x_0      & := 0,                             \\
    f_0      & := 8\epsilon^2 + \frac{4}{1 - p},
  \end{align}
\end{subequations}
we define, for all \(k \in \{1, \ldots, k_{\epsilon}\}\),
\begin{subequations}%
  \label{eq:def-fk-gk}
  \begin{align}
    B_k   & := k^p,
    \label{eq:val-Bk}%
    \\
    x_{k} & := x_{k-1} + s_{k-1},         \\
    f_{k} & := f_{k-1} + g_{k-1} s_{k-1},
    \label{eq:def-fk}
  \end{align}
\end{subequations}
and for all \(k \in \{0, \ldots, k_{\epsilon}\}\),
\begin{subequations}%
  \label{eq:def-sk-nuk}
  \begin{align}
    s_k   & := -B_k^{-1} g_k > 0,
    \label{eq:sk-val}%
    \\
    \nu_k & := \frac{1}{\alpha^{-1}\Delta_k^{-1} + |B_k|(1+\alpha^{-1}\Delta_k^{-1})}.
    \label{eq:nuk-val}
  \end{align}
\end{subequations}
As in \citep[Theorem~\(2.2.3\)]{cartis-gould-toint-2022}, the sequences~\eqref{eq:def-wk-gk},~\eqref{eq:def-fk-gk} and~\eqref{eq:def-sk-nuk} are created specifically so that we may generate iterates that satisfy the assumptions of \Cref{prop:hermite-interpol}, along with \(\nu_k^{-1/2} \xi_{\mathrm{cp}}(\Delta_k; x_k, \nu_k)^{1/2} = |g_k| > \epsilon\) for \(k \in \{0, \ldots, k_{\epsilon} - 1\}\), and \(|g_{k_{\epsilon}}| = \epsilon\).
It is worth noticing that we chose \(\{B_k\}\) so that~\Cref{asm:hess-approx-growth} is satisfied if every iteration is successful (which is shown in the proof of \Cref{th:slow-cv}), and that \(k_{\epsilon} = O(\epsilon^{-2/(1-p)})\).

First, \Cref{lem:f-bnd} establishes bounds on \(f_k\).

\begin{lemma}%
  \label{lem:f-bnd}
  Using the parameters in~\eqref{eq:def-eps-p-keps} and the sequences defined in~\eqref{eq:def-wk-gk},~\eqref{eq:def-fk-gk}, and~\eqref{eq:def-sk-nuk}, the following properties hold for the sequence \(\{f_k\}\):
  \begin{enumerate}
    \item for all \(k \in \{1, \ldots, k_{\epsilon}\}\),
      \begin{equation}
        f_k < f_{k-1},
      \end{equation}
    \item for all \(k \in \{0, \ldots, k_{\epsilon}\}\),
      \begin{equation}
        \label{eq:keps-bnd}
        0 \le f_0 - f_k \le 4 \epsilon^2 \left(2 + \frac{k^{(1-p)}}{1 - p} \right) \le 8\epsilon^2 + \frac{4}{1 - p},
      \end{equation}
    \item for all \(k \in \{0, \ldots, k_{\epsilon}\}\),
      \begin{equation}
        \label{eq:fk-pos}
        f_k \ge 0.
      \end{equation}
  \end{enumerate}
\end{lemma}

\begin{proof}
  First, we notice that for all \(k \in \{0, \ldots, k_{\epsilon}\}\), \(g_k < 0\) and \(s_k > 0\).
  By combining these observations and the definition of \(f_k\), we deduce that \(f_k < f_{k-1}\) for all \(k \in \{1, \ldots, k_{\epsilon}\}\), and in particular
  \begin{equation*}
    f_0 - f_k \ge 0.
  \end{equation*}
  Inequalities~\eqref{eq:keps-bnd} hold for \(k = 0\) and for \(k = 1\) because \(f_0 - f_1 = -g_0 s_0 = 4 \epsilon^2\).
  For all \(k \in \{2, \ldots, k_{\epsilon}\}\),
  \begin{equation*}
    \begin{aligned}
      f_0 - f_{k} & = - \sum_{i=0}^{k-1}g_i s_i                                         \\
                  & = -g_0 s_0 + \sum_{i=1}^{k-1} g_i^2 i^{-p}                          \\
                  & = 4 \epsilon^2 + \sum_{i=1}^{k-1} \epsilon^2(1+w_i)^2 i^{-p}        \\
                  & = \epsilon^2 \left(4 + \sum_{i=1}^{k-1} (1 + w_i)^2 i^{-p} \right).
    \end{aligned}
  \end{equation*}
  Now,
  \begin{equation*}
    \begin{aligned}
      \sum_{i=1}^{k-1} (1 + w_i)^2 i^{-p} & \le \sum_{i=1}^{k-1} 4 i^{-p}                                   &  & \text{because \(1 + w_i \le 2\)}                                                   \\
                                          & \le 4 \left(1 + \sum_{i=2}^{k-1} i^{-p} \right)                                                                                                         \\
                                          & \le 4 \left(1 + \sum_{i=2}^{k-1} \int_{i-1}^i t^{-p} dt \right) &  & \text{because \(i^{-p} = \int_{i-1}^{i} i^{-p} dt \le \int_{i-1}^{i} t^{-p} dt\) } \\
                                          & \le 4 \left(1 + \int_1^{k-1} t^{-p} dt \right)                                                                                                          \\
                                          & \le 4 \left(1 + \int_1^{k} t^{-p} dt \right)                                                                                                            \\
                                          & = 4 \left(1 + \frac{k^{1-p}-1}{1-p}\right)                                                                                                              \\
                                          & \le 4 \left(1 + \frac{k^{1-p}}{1-p}\right).
    \end{aligned}
  \end{equation*}
  This results in
  \begin{equation}
    \label{eq:maj-f0-fk}
    f_0 - f_k \le 4 \epsilon^2 + 4 \epsilon^2 \left(1 + \frac{k^{1-p}}{1 - p} \right) = 8 \epsilon^2 + 4 \frac{\epsilon^2 k^{1-p}}{1 - p}.
  \end{equation}
  Finally, since \(k \le k_{\epsilon} = \lfloor \epsilon^{-2/(1-p)} \rfloor \le \epsilon^{-2/(1-p)}\), we have, for all \(k \le k_{\epsilon}\),
  \begin{equation}
    \label{eq:k-maj}
    \epsilon^2 k^{(1-p)} \le 1.
  \end{equation}
  We combine~\eqref{eq:maj-f0-fk} and~\eqref{eq:k-maj} to obtain~\eqref{eq:keps-bnd}.
  The value of \(f_0\) and~\eqref{eq:keps-bnd} then allows us to establish~\eqref{eq:fk-pos}.
\end{proof}

Now, \Cref{lem:g-bnd} establishes a bound for \(|g_{k+1} - g_k|\).

\begin{lemma}%
  \label{lem:g-bnd}
  Using the parameters in~\eqref{eq:def-eps-p-keps} and the sequences defined in~\eqref{eq:def-wk-gk},~\eqref{eq:def-val0} and~\eqref{eq:def-sk-nuk}, we have that, for all \(k \in \{0, \ldots, k_{\epsilon}\}\),
  \begin{equation}
    |g_{k+1} - g_k| \le s_k.
  \end{equation}
\end{lemma}

\begin{proof}
  For \(k \in \{0, \ldots, k_{\epsilon} - 1\}\),
  \begin{equation}
    \label{eq:gk-maj}
    |g_{k+1} - g_k| = |-\epsilon(1 + w_{k+1}) + \epsilon (1 + w_k)| = \epsilon / k_{\epsilon}.
  \end{equation}
  Since \(p < 1\) and \(k < k_{\epsilon}\), we have \(k^p / k_{\epsilon} \le 1 \le 1 + w_k\).
  We multiply the latter inequality by \(\epsilon k^{-p}\) to obtain \(\epsilon / k_{\epsilon} \le k^{-p} \epsilon (1 + w_k)\), which leads to \(|g_{k+1} - g_k| \le s_{k}\) using~\eqref{eq:gk-maj}.
\end{proof}

The following result uses \Cref{lem:f-bnd} and \Cref{lem:g-bnd} to apply \Cref{prop:hermite-interpol}.
\begin{proposition}%
  \label{prop:f-interpol}
  Using the parameters in~\eqref{eq:def-eps-p-keps} and the sequences defined in~\eqref{eq:def-wk-gk},~\eqref{eq:def-val0} and~\eqref{eq:def-sk-nuk}, there exists \(f: \R \to \R\) continuously differentiable such that
  \begin{equation}
    f(x_k) = f_k, \quad f'(x_k) = g_k.
  \end{equation}
  In addition, the assumptions of \Cref{prop:hermite-interpol} hold, so that \(|f|\) and \(|f'|\) are bounded by a constant independent of \(k\).
\end{proposition}

\begin{proof}
  We can see that \(s_k > 0\) and, by definition of \(f_k\),
  \begin{equation*}
    |f_{k+1} - (f_k + g_k s_k)| = 0.
  \end{equation*}
  \Cref{lem:g-bnd} shows that
  \begin{equation*}
    |g_{k+1} - g_k| \le s_k.
  \end{equation*}
  Using \Cref{lem:f-bnd}, we know that for all \(k \in \{0, \ldots, k_{\epsilon}\}\), \(f_k \ge 0\), and since \(\{f_k\}\) is decreasing, we have
  \begin{equation*}
    |f_k| \le f_0.
  \end{equation*}
  In addition,
  \begin{equation*}
    |g_k| \le 2 \epsilon \le 1 \quad \text{and} \quad s_k \le |g_k| \le 1.
  \end{equation*}
  The result follows from \Cref{prop:hermite-interpol}.
\end{proof}

For the following lemma, we define the sequence \(\{s_{k,1}\}\) such that for all \(k \in \{0, \ldots, k_{\epsilon}\}\),
\begin{equation}
  \label{eq:sk1-sharp}
  s_{k,1} := -\nu_k g_k.
\end{equation}

\begin{lemma}%
  \label{lem:sk-inequality}
  Using the parameters in~\eqref{eq:def-eps-p-keps} and the sequences defined in~\eqref{eq:def-wk-gk},~\eqref{eq:def-val0} and~\eqref{eq:def-sk-nuk}, we establish that, for all \(k \in \{0, \ldots, k_{\epsilon}\}\) such that \(\Delta_k \ge 1\),
  \begin{equation}
    \label{eq:sk-maj-deltak-sk1}
    |s_k| \le \min (\Delta_k, \beta |s_{k, 1}|).
  \end{equation}
\end{lemma}
\begin{proof}
  On the one hand, we have
  \begin{equation}
    \label{eq:sk-maj-deltak}
    |s_k| = \epsilon \frac{(1 + w_k)}{B_k} \le 2 \epsilon \le 1 \le \Delta_k.
  \end{equation}
  On the other hand, since \(B_k^{-1} \le 1\) and \(\Delta_k \ge 1\),
  \begin{equation*}
    2 \alpha^{-1} + 1 \ge \alpha^{-1} \Delta_k^{-1}(B_k^{-1} + 1) + 1,
  \end{equation*}
  so that
  \begin{equation*}
    1 \le \frac{2 \alpha^{-1} + 1}{\alpha^{-1} \Delta_k^{-1}(B_k^{-1} + 1) + 1} \le \frac{\beta}{\alpha^{-1} \Delta_k^{-1}(B_k^{-1} + 1) + 1}.
  \end{equation*}
  We multiply the above inequality by \(B_k^{-1}\) to obtain
  \begin{equation*}
    B_k^{-1} \le \frac{\beta B_k^{-1}}{\alpha^{-1} \Delta_k^{-1}(B_k^{-1} + 1) + 1} = \frac{\beta}{\alpha^{-1} \Delta_k^{-1} + B_k(1 + \alpha^{-1} \Delta_k^{-1})} = \beta \nu_k,
  \end{equation*}
  and, by multiplying by \(|g_k|\), we deduce that
  \begin{equation}
    \label{eq:sk-maj-sk1}
    |s_k| = B_k^{-1} |g_k| \le \beta \nu_k |g_k| = \beta |s_{k,1}|.
  \end{equation}
  We combine~\eqref{eq:sk-maj-deltak} and~\eqref{eq:sk-maj-sk1} to obtain~\eqref{eq:sk-maj-deltak-sk1}.
\end{proof}

The following theorem finally establishes the main result of this section.

\begin{theorem}[Slow convergence of \Cref{alg:tr-nonsmooth}]%
  \label{th:slow-cv}
  \Cref{alg:tr-nonsmooth} applied to~\eqref{eq:nlp} with model \(m_k\) satisfying \Cref{asm:models}, \Cref{asm:model-diff}, \Cref{asm:model-ubd} and using Hessian approximations \(\{B_k\}\) satisfying \Cref{asm:hess-approx-growth} may require as many as \(O(\epsilon^{-2 / (1 - p)})\) iterations to produce an iterate \(x_{k_{\epsilon}}\) such that
  \begin{equation}
    \label{eq:eps-stat-smooth}
    \nu_{k_{\epsilon}}^{-1/2} \xi_{\mathrm{cp}}({\Delta}_{k_{\epsilon}}; x_{k_{\epsilon}}, \nu_{k_{\epsilon}})^{1/2} \le \epsilon,
  \end{equation}
  in the sense that there exists \(f: \R \to \R\) satisfying the assumptions of \Cref{lem:num-successful-unbounded} and for which~\eqref{eq:eps-stat-smooth} occurs for the first time after \(k_\epsilon\) iterations.
\end{theorem}

\begin{proof}
  The proof consists in constructing \(f:\R \to \R\) by interpolation, as in \citep[Theorem~\(2.2.3\)]{cartis-gould-toint-2022}.
  Let \(n = 1\), \(h = 0\), \(\ell = -\infty\), \(u = +\infty\).
  We use the parameters in~\eqref{eq:def-eps-p-keps} and the sequences defined in~\eqref{eq:def-wk-gk},~\eqref{eq:def-val0} and~\eqref{eq:def-sk-nuk}.
  We invoke \Cref{prop:f-interpol} to obtain \(f: \R \to \R\) differentiable and bounded such that \(f(x_k) = f_k\) and \(f'(x_k) = g_k\).
  Our goal is to show that \(\{x_k\}\), \(\{s_k\}\), \(\{f_k\}\) and \(\{g_k\}\) satisfy all our assumptions and are generated by \Cref{alg:tr-nonsmooth} applied to \(f\) with \(x_0 = 0\) and with the special value of \(\{B_k\}\) in~\eqref{eq:val-B0} and~\eqref{eq:val-Bk}.

  We proceed by choosing \(0 \le k \le k_{\epsilon}\) such that \(\Delta_k \ge 1\), which holds at least for \(k = 0\), and going through the steps of \Cref{alg:tr-nonsmooth} at iteration \(k\) to check that it generates the iterates defined in~\eqref{eq:def-wk-gk},~\eqref{eq:def-val0} and~\eqref{eq:def-sk-nuk}.

  In \Cref{step:nuk-interval}, \(\nu_k\) in~\eqref{eq:nuk-val} is as large as allowed.

  In \Cref{step:sk1-computation}, \Cref{lem:sol-pb-B-gen} indicates that \(s_{k,1}\) in~\eqref{eq:sk1-sharp} is a global minimizer of~\eqref{eq:def-m-nu} with \(\psi = 0\).
  As \(1 + w_k \le 2\) and \(|B_k| \ge 1\), we observe that
  \begin{equation*}
    |s_{k,1}| = |\nu_k g_k| = \frac{\epsilon (1 + w_k)}{\alpha^{-1}\Delta_k^{-1} + |B_k|(1 + \alpha^{-1}\Delta_k^{-1})} \le 2\epsilon \le 1 \le \Delta_k,
  \end{equation*}
  which implies that \(s_{k,1}\) is a solution of~\eqref{eq:ipg-iter-sk1} because the condition \(|s_{k,1}| \le \Delta_k\) is already satisfied.

  In \Cref{step:sk-computation}, let \(m(\cdot; x_k, B_k)\) be defined as in~\eqref{eq:def-model-quad}.
  \(m(\cdot; x_k, B_k)\) satisfies \Cref{asm:models}, and using \Cref{lem:sol-pb-B-gen}, we have that \(s_k\) in~\eqref{eq:sk-val} with \(\psi = 0\) and \(B = B_k\) is its global minimizer.
  \Cref{lem:sk-inequality} shows that
  \begin{equation*}
    |s_k| \le \min (\Delta_k, \beta |s_{k, 1}|),
  \end{equation*}
  which also implies that \(s_k\) is a solution of~\eqref{eq:def-model-quad}.

  In \Cref{step:rhok-computation}, we compute
  \begin{equation}%
    \label{eq:rho-k-2val}
    \begin{aligned}
      \rho_k & = \frac{f_k - f_{k+1}}{m(0; x_k, B_k) - m(s_k; x_k, B_k)} \\
             & = \frac{f_k - f_{k+1}}{f_k - f_k - g_k s_k - B_k s_k^2/2} \\
             & = \frac{f_k - f_{k+1}}{g_k^2 B_k^{-1}/2}                  \\
             & = \frac{-g_k s_k}{g_k^2 B_k^{-1}/2}                       \\
             & = \frac{B_k^{-1}g_k^2}{g_k^2 B_k^{-1}/2}                  \\
             & = 2.
    \end{aligned}
  \end{equation}

  In \Cref{setp:xk-update}, \(\rho_k = 2\) implies that \(x_{k+1} = x_k + s_k\), and in \Cref{step:deltak-update}, we can set \(\Delta_{k+1} = \min(\gamma_3 \Delta_k, \, \Delta_{\max}) \ge \Delta_k \ge 1\).

  Now, either \(\nu_k^{-1/2} \xi_{\mathrm{cp}}(\Delta_k; x_k, \nu_k)^{1/2} > \epsilon\), and we perform the next iteration of \Cref{alg:tr-nonsmooth}, or \(\nu_k^{-1/2} \xi_{\mathrm{cp}}(\Delta_k; x_k, \nu_k)^{1/2} \le \epsilon\), which stops the algorithm.
  We have shown that \(s_{k,1}\) is a solution of~\eqref{eq:ipg-iter-sk1}, thus
  \begin{equation}
    \label{eq:x1-eq-smooth}
    \xi_{\mathrm{cp}}(\Delta_k; x_k, \nu_k) = f_k - (f_k + g_k s_{k,1}) = -g_k s_{k,1} = \nu_k g_k^2,
  \end{equation}
  and
  \begin{equation}
    \label{eq:nuinv-xi-grad}
    \nu_k^{-1/2} \xi_{\mathrm{cp}}(\Delta_k; x_k, \nu_k)^{1/2} = |g_k|.
  \end{equation}
  Therefore, for all \(k \in \{0, \ldots, k_{\epsilon} - 1\}\), \(\nu_k^{-1/2} \xi_{\mathrm{cp}}(\Delta_k; x_k, \nu_k)^{1/2} > \epsilon\), and \(\nu_{k_{\epsilon}}^{-1/2} \xi_{\mathrm{cp}}(\Delta_{k_{\epsilon}}; x_{k_{\epsilon}}, \nu_{k_{\epsilon}})^{1/2} = \epsilon\), so that \Cref{alg:tr-nonsmooth} performs exactly \(k_{\epsilon}\) iterations to generate \(x_{k_{\epsilon}}\) satisfying~\eqref{eq:eps-stat-smooth}.

  To finish the proof, we must verify that \Cref{asm:model-diff}, \Cref{asm:model-ubd} and \Cref{asm:hess-approx-growth} hold.
  \Cref{asm:model-diff} is satisfied thanks to \Cref{prop:justif-asm-model-diff}.
  \Cref{asm:model-ubd} is satisfied with \(\kappa_{\mathrm{ubd}} = \tfrac{1}{2}\) because
  \begin{equation*}
    |f_{k+1} - m(s_k; x_k, B_k)| = |f_{k+1} - f_k - g_k s_k - \tfrac{1}{2}B_k s_k^2| = \tfrac{1}{2} B_k s_k^2 \le \tfrac{1}{2} (1 + B_k) s_k^2.
  \end{equation*}
  Finally, our choice of \(B_k\) allows \Cref{asm:hess-approx-growth} to be satisfied because all iterations are successful and \(\sigma_k = k\).
\end{proof}

\section{Numerical verification of the bound}%
\label{sec:num-verif}

We construct \(f: \R \to \R\) satisfying the properties of the function in the proof of \Cref{th:slow-cv}.
The construction follows the formula used in the proof of \citep[Theorem~\(A.9.2\)]{cartis-gould-toint-2022}, and we use similar notation.

We use again the parameters~\eqref{eq:def-eps-p-keps}, and the sequences~\eqref{eq:def-wk-gk}--\eqref{eq:def-sk-nuk}.
Define the cubic Hermite interpolant
\begin{equation}
  \pi_k(\tau) \coloneqq c_{k,0} + c_{k,1} \tau + c_{k,2} \tau^2 + c_{k,3} \tau^3,
\end{equation}
where, for all \(k \in \{0, \ldots, k_\epsilon\}\), \(c_{k,0} = f_k\), \(c_{k,1} = g_k\), and \(c_{k,2}\), \(c_{k,3}\) solve
\begin{equation}
  \label{eq:syst-c}
  \begin{bmatrix}
    s_k^2 & s_k^3  \\
    2s_k  & 3s_k^2
  \end{bmatrix}
  \begin{bmatrix}
    c_{k, 2} \\
    c_{k, 3}
  \end{bmatrix}
  =
  \begin{bmatrix}
    f_{k+1} - (f_k + g_k s_k) \\
    g_{k+1} - g_k
  \end{bmatrix}
  =
  \begin{bmatrix}
    0 \\
    g_{k+1} - g_k
  \end{bmatrix}
  .
\end{equation}
We use the additional conditions \(f_{-1} = f_0\), \(g_{-1} = 0\), \(f_{k_{\epsilon} + 1} = f_{k_{\epsilon}}\), \(g_{k_{\epsilon} + 1} = g_{k_{\epsilon}}\), and \(x_{-1} = -s_{-1}\), where \(s_{-1} = 1\), which allows~\eqref{eq:fk-gk-ineq} to hold with \(\kappa_f = 1\), because \(|f_0 - (f_{-1} + g_{-1} s_{-1})| = 0\), and \(|g_0 - g_{-1}| = |g_0| = \epsilon (1 + w_0) = 2 \epsilon \le 1 = s_{-1}\) since \(\epsilon \le 1/2\).
Finally,
\begin{equation}%
  \label{eq:def-f}
  f(x) \coloneqq
  \begin{cases}
    f_0              & \text{if } x \le x_{-1}                                                         \\
    \pi_k(x - x_k)   & \text{if } x \in (x_k; x_{k+1}] \text{ for } k \in \{-1, \ldots, k_{\epsilon}\} \\
    f_{k_{\epsilon}} & \text{if } x > x_{k_{\epsilon}} + s_{k_{\epsilon}}.
  \end{cases}
\end{equation}

By construction, \(f\) is a piecewise polynomial of degree \(3\).
We have \(\pi_k(0) = f_k\), \(\pi_k'(0) = g_k\), \(\pi_k(s_k) = f_{k+1}\) thanks to the definition of \(f\) in~\eqref{eq:def-fk} and the first line of~\eqref{eq:syst-c}, and \(\pi_k'(s_k) = g_{k+1}\) with the second line of~\eqref{eq:syst-c}.
Thus, \(f: \R \rightarrow \R\) is continuously differentiable over \((x_{-1}, x_{k_\epsilon + 1})\).

We minimize \(f\) using \Cref{alg:tr-nonsmooth} as implemented in \citep{baraldi-orban-regularized-optimization-2022}, without nonsmooth regularizer, and with starting point \(x_0 = 0\).
Inside TR, we set \(B_k = k^p\) so that \(\{B_k\}\) grows unbounded and \Cref{asm:hess-approx-growth} holds, because \(\rho_k = 2\) in~\eqref{eq:rho-k-2val} so that all iterations are very successful.
In \Cref{step:sk-computation}, we use the analytical solution \(s_k = -B_k^{-1} \nabla f(x_k)\) of~\eqref{eq:pb-B-gen} given by \Cref{lem:sol-pb-B-gen} in order to avoid rounding errors occurring in a subproblem solver for~\eqref{eq:model-k}.
This expression of \(s_k\) satisfies the trust-region constraint by construction thanks to \Cref{lem:sk-inequality}.
The modified TR implementation is available from \url{https://github.com/geoffroyleconte/RegularizedOptimization.jl/tree/unbounded}.

We set \(p = 1/10\), \(\alpha = \beta = 10^{+16}\), \(\gamma_3 = 3\), \(\Delta_{\max} = 10^3\) and \(\epsilon = 1/10\), so that \(k_{\epsilon} = 166\).
We observe that TR converges in precisely \(166\) iterations.
With \(\epsilon = 1/20\), we obtain the convergence of TR in precisely \(k_\epsilon = 778\) iterations.

In order to make the oscillations of \(f'\) clearly visible, \Cref{fig:f-g} shows plots of \(f\) and \(f'\) over \([0, x_{k_{\epsilon}+1}]\) with \(\epsilon = 1/3\).
\Cref{tbl:theor-vals-crit}  shows the theoretical values of \(\nu_k^{-1/2}\xi_{\mathrm{cp}}(\Delta_k; x_k, \nu_k)^{1/2} = |g_k|\) according to~\eqref{eq:nuinv-xi-grad}.
TR converges in \(11\) iterations and produces the logs in \Cref{fig:tr-logs} that align with these theoretical values.
Note that \(\rho_k = 2\), as predicted by~\eqref{eq:rho-k-2val}, and therefore, that each iteration is successful.

\begin{table}[H]%
  \centering
  \caption{%
    \label{tbl:theor-vals-crit}
    Rounded theoretical values of \(\nu_k^{-1/2}\xi_{\mathrm{cp}}(\Delta_k; x_k, \nu_k)^{1/2}\) for \(\epsilon = 1/3\).
  }
  \begin{tabular}{|c||c|c|c|c|c|c|c|c|c|c|c|c|}%
    \hline
    \(k\)                                                         & \(0\)    & \(1\)    & \(2\)    & \(3\)    & \(4\)    & \(5\)    & \(6\)    & \(7\)    & \(8\)    & \(9\)    & \(10\)   & \(11\)   \\
    \hline
    \(\nu_k^{-1/2}\xi_{\mathrm{cp}}(\Delta_k; x_k, \nu_k)^{1/2}\) & \(0.67\) & \(0.64\) & \(0.61\) & \(0.58\) & \(0.55\) & \(0.52\) & \(0.48\) & \(0.45\) & \(0.42\) & \(0.39\) & \(0.36\) & \(0.33\) \\
    \hline
  \end{tabular}
\end{table}

\begin{figure}[H]%
  \centering
  \caption{%
    \label{fig:tr-logs}
    TR logs with \(\epsilon = 1/3\).
    \emph{outer} denotes the iteration number, \emph{inner} is the number of iterations performed by the subsolver to solve~\eqref{eq:model-k} with the model in~\eqref{eq:def-model-quad}, \emph{\(\sqrt{}\xi cp / \sqrt{} \nu\)} is \(\nu_k^{-1/2}\xi_{\mathrm{cp}}(\Delta_k; x_k, \nu_k)^{1/2}\), \emph{\(\sqrt{}\xi\)} is the numerator of~\eqref{eq:rhok}, \emph{\(\|s\|\)} is \(\|s_k\|\), and the remaining columns refer unambiguously to data used in \Cref{alg:tr-nonsmooth}.
  }
  \begin{jllisting}
outer    inner     f(x)     h(x) √ξcp/√ν      √ξ        ρ       Δ     ‖x‖     ‖s‖    ‖Bₖ‖
    1        1  5.3e+00  0.0e+00 6.7e-01 4.7e-01  2.0e+00 1.0e+00 0.0e+00 6.7e-01 1.0e+00
    2        1  4.9e+00  0.0e+00 6.4e-01 4.5e-01  2.0e+00 3.0e+00 6.7e-01 6.4e-01 1.0e+00
    3        1  4.5e+00  0.0e+00 6.1e-01 4.1e-01  2.0e+00 9.0e+00 1.3e+00 5.7e-01 1.1e+00
    4        1  4.1e+00  0.0e+00 5.8e-01 3.9e-01  2.0e+00 2.7e+01 1.9e+00 5.2e-01 1.1e+00
    5        1  3.8e+00  0.0e+00 5.5e-01 3.6e-01  2.0e+00 8.1e+01 2.4e+00 4.7e-01 1.1e+00
    6        1  3.6e+00  0.0e+00 5.2e-01 3.4e-01  2.0e+00 2.4e+02 2.9e+00 4.4e-01 1.2e+00
    7        1  3.4e+00  0.0e+00 4.8e-01 3.1e-01  2.0e+00 7.3e+02 3.3e+00 4.1e-01 1.2e+00
    8        1  3.2e+00  0.0e+00 4.5e-01 2.9e-01  2.0e+00 1.0e+03 3.7e+00 3.7e-01 1.2e+00
    9        1  3.0e+00  0.0e+00 4.2e-01 2.7e-01  2.0e+00 1.0e+03 4.1e+00 3.4e-01 1.2e+00
   10        1  2.8e+00  0.0e+00 3.9e-01 2.5e-01  2.0e+00 1.0e+03 4.4e+00 3.2e-01 1.2e+00
   11        1  2.7e+00  0.0e+00 3.6e-01 2.3e-01  2.0e+00 1.0e+03 4.7e+00 2.9e-01 1.3e+00
   12        1  2.6e+00  0.0e+00 3.3e-01                  1.0e+03 5.0e+00 2.6e-01 1.3e+00
TR: terminating with √ξcp/√ν = 0.3333333333333333
"Execution stats: first-order stationary"
  \end{jllisting}
\end{figure}

\begin{figure}%
  \centering
  \includetikzgraphics{f-plot-1-3}
  \hspace{2em}  
  \includetikzgraphics{g-plot-1-3}
  \\
  \hspace{1.1em}
  \includetikzgraphics{x-plot-1-3}
  \hfill
  \includetikzgraphics{s-plot-1-3}
  \caption{%
    \label{fig:f-g}
    Illustration of example~\eqref{eq:def-f} with \(\epsilon = 1/3\).
    Top row: values of \(f\) (left) and of \(f'\) (right) for \(x \in [0, x_{k_{\epsilon}+1}]\).
    Bottom row: iterates \(x_k\) (left) and steps \(s_k\) (right) for \(k \in [0, k_{\epsilon}+1]\).
  }
\end{figure}

The code to run this experiment is available at \url{https://github.com/geoffroyleconte/docGL/blob/master/regularized-opt/test-unbounded-hess.jl}.
By making similar changes to the algorithm TRDH \citep{leconte-orban-2024}, which can be found at the same URL, we obtain the same number of iterations.

\section{Discussion}%
\label{sec:discussion}

We have shown that it is possible to establish convergence and sharp worst-case evaluation complexity of \Cref{alg:tr-nonsmooth} in the presence of unbounded Hessian approximations \(B_k\), provided they do not grow too fast---c.f., \Cref{asm:hess-approx-growth}.
We established that the complexity bound can be attained, and we gave an example of a function for which it was attained, both theoretically and numerically.

\citet{aravkin-baraldi-orban-2022} compare the performance of \Cref{alg:tr-nonsmooth} to other frameworks, but use a formula for \(\nu_k\) that assumes that \(\{B_k\}\) remains bounded.
Their implementation uses limited-memory SR1 and BFGS approximations.
As it happens, such limited-memory approximations do remain bounded under standard assumptions; see \citep{burdakov-gong-zikrin-yuan-2017} for LBFGS\@.
The fact that LSR1 approximations remain bounded was not known to us at the time of writing \citep{aravkin-baraldi-orban-2022}.
However, an early version of that manuscript contained a procedure to maintain bounds on the extreme eigenvalues of such an approximation, and skip the update if those bounds became too large---see Section~\(4.2\) in \url{https://arxiv.org/pdf/2103.15993v1}.
We only realized later that that very analysis of the extreme eigenvalues shows that LSR1 approximations remain bounded provided that the sequence of initial matrices remains bounded, which is the case in the experiments of \citep{aravkin-baraldi-orban-2022}.

When \(p = 1\) in \Cref{asm:hess-approx-growth} or the growth of \(\|B_k\|\) is not governed by the number of successful iterations, it may still be possible to establish convergence in the sense that \(\liminf \nu_k^{-1/2} \xi_{\mathrm{cp}}(\Delta_k; x_k, \nu_k) = 0\) as in \citep[\S\(8.4.1.2\)]{conn-gould-toint-2000}, where the main assumption is~\eqref{eq:asm-sum-invmaxBj-inf}.
Generalizations of \Cref{asm:hess-approx-growth} might replace \(\sigma_k\) with \(k\), to account for situations where model Hessians are updated on unsuccessful iterations, or by a positive function \(\phi(\sigma_k)\) or \(\phi(k)\).
In view of~\eqref{eq:asm-sum-invmaxBj-inf}, such \(\phi\) would have to satisfy
\[
  \sum_{k=0}^{\infty} \frac{1}{1 + \max_{0 \leq j \leq k} \phi(j)^p} = \infty.
\]
Under the simplifying, but reasonable, assumption that \(\phi\) is continuous and nondecreasing, it would be necessary and sufficient that
\[
  \int_1^{\infty} \frac{1}{1 + \phi(t)^p} \, \mathrm{d}t = \infty.
\]
We expect that sharp worst-case evaluation complexity bounds also hold for such more general cases.

Another possible extension of the present work would be to analyze the worst-case evaluation complexity of AR\(p\)-type methods in the presence of potentially unbounded model Hessians.

Although \Cref{alg:tr-nonsmooth} does not reduce to the ``standard'' trust-region method in the case where \(h = 0\)---by which we mean, e.g., the basic trust-region algorithm of~\cite[Chapter~\(6\)]{conn-gould-toint-2000}---we expect that the techniques of the present paper can be used under \Cref{asm:hess-approx-growth}, or generalizations thereof, to establish similar complexity bounds.
Whether or not quasi-Newton updates satisfy \Cref{asm:hess-approx-growth} under certain assumptions is the subject of ongoing research.

\small
\subsection*{Acknowledgements}

We express our sincere gratitude to three anonymous referees and the associate editor, whose insightful questions, comments and suggestions significantly improved this research.

\subsection*{Competing interests}

We certify that the research submitted here is original, is our own, and is not being evaluated elsewhere for publication.
This work was supported by an NSERC Discovery grant.

\subsection*{Data availability}

The code used to produce the numerical results is available from \url{https://github.com/geoffroyleconte/unbounded-hessian-code}.
The solvers are available from \url{https://github.com/geoffroyleconte/RegularizedOptimization.jl/tree/unbounded}.

\bibliographystyle{abbrvnat}
\bibliography{abbrv,unbounded-hessian}
\end{document}